\documentclass{m2an}
\usepackage{amsmath,amsfonts,amssymb}
\usepackage{stmaryrd}
\usepackage{latexsym}
\usepackage{epsfig,overpic}
\usepackage{subfigure}
\usepackage{color}
\usepackage[hidelinks]{hyperref}
\let\ep\varepsilon

\newcommand{\ahn}{\hyperref[e:anh]{a_h^n}}

\newcommand{\bn}{\mathbf n}
\newcommand{\bp}{\mathbf p}

\newcommand{\bw}{\mathbf w}
\newcommand{\bx}{\mathbf x}
\newcommand{\by}{\mathbf y}

\newcommand{\we}{{\bw^e}}

\renewcommand{\l}[1]{{#1}^l}
\newcommand{\ul}{\l{u}}

\newcommand{\refts}{$\!\!\!L_t\downarrow\! \setminus L_x \! \rightarrow$\hspace*{-0.6cm}}

\newcommand{\winfn}{{\bw_{\infty}^{\bn}}}

\usepackage{scalerel}
\newcommand{\deltah}{{\delta_{\scaleto{h}{2.75pt}}}}
\newcommand{\dist}{\operatorname{dist}}
\newcommand{\meas}{\operatorname{meas}}
\newcommand{\wOm}{\widetilde{\Omega}}
\newcommand{\Om}[2]{\Omega^{#1}_{#2}}

\newcommand{\TS}[1]{\mathcal{T}_{\mathcal{S}^\pm}^{#1}}
\newcommand{\TSp}[1]{\mathcal{T}_{\mathcal{S}^+}^{#1}}
\newcommand{\Td}[1]{\mathcal{T}_{\delta}^{#1}}
\newcommand{\Sd}[1]{{\mathcal{S}_{\deltah}^\pm(\Omega^{#1}) }}
\newcommand{\Sdp}[1]{{\mathcal{S}_{\deltah}^+(\Omega^{#1}) }}
\newcommand{\Sdh}[1]{{\mathcal{S}_{\deltah}^\pm(\Omega_h^{#1}) }}
\newcommand{\Sdhp}[1]{{\mathcal{S}_{\deltah}^+(\Omega_h^{#1}) }}
\renewcommand{\S}{{\mathcal{S}_\delta}}

\newcommand{\Odn}[1]{{\mathcal{O}_{\deltah}(\Omega^{#1})}}
\newcommand{\Odhn}[1]{{\mathcal{O}_{\deltah}\!(\Omega_h^{#1})}}
\newcommand{\Odt}[1]{{\mathcal{O}_{\deltah,\mathcal{T}}^{#1}}}
\newcommand{\eoc}[1]{$\text{eoc}_{\texttt{#1}}$}

\newcommand{\jump}[1]{[\![ #1 ]\!]}
\newcommand{\norm}[1]{\Vert #1 \Vert}

\newcommand{\enorm}[1]{\vert\!\vert\!\vert #1 \vert\!\vert\!\vert}
\newcommand{\innerprod}[2]{\left( #1, #2 \right)}

\newcommand{\consist}{\mathcal{E}_C^n}
\newcommand{\interpol}{\mathcal{E}_I^n}
\newcommand{\err}{\mathbb{E}}

\newcommand{\E}[1]{\mathcal{E}_{#1}}

\newcommand{\T}{\mathcal T}
\newcommand{\Div}{\operatorname{\rm div}}

\newcommand{\rr}{\mathbb{R}}

\newcommand{\Qs}{\mathcal{Q}}
\newcommand{\Qsn}{\mathcal{Q}}

\DeclareGraphicsExtensions{.pdf,.eps,.ps,.eps.gz,.ps.gz,.eps.Y}
\def\Fh{\mathcal{F}_h}

\usepackage{booktabs}
\usepackage{siunitx}
\newcommand{\nbf}[1]{{\underline{\num{#1}}}}
\newcommand{\nit}[1]{\underline{\underline{\num{#1}}}}
\renewcommand{\O}{{\mathbf{\mathcal{O}_{\delta}}}}
\newcommand{\numQ}[1]{\num[round-precision=2,round-mode=places, scientific-notation=false]{#1}}

\newtheorem{assumption}{Assumption}[section]
\newtheorem{remark}{Remark}[section]
\newtheorem{lemma}{Lemma}[section]
\newtheorem{theorem}{Theorem}[section]

\begin{document}
\title{An Eulerian finite element method for PDEs in time-dependent domains}\thanks{C.L. was partially supported by the German Science Foundation (DFG) within the project ``LE 3726/1-1''}\thanks{M.O. was partially supported by NSF through the Division of Mathematical Sciences grant 1717516}
\author{Christoph Lehrenfeld}\address{Institute for Numerical and Applied Mathematics, University of G\"ottingen, G\"ottingen, Germany, \\ \email{lehrenfeld@math.uni-goettingen.de}%
}
\author{Maxim A. Olshanskii}\address{Department of Mathematics, University of Houston, Houston, Texas 77204-3008, \\ \email{molshan@math.uh.edu}
}
\date{\today}

\begin{abstract}
The paper introduces a new finite element numerical method for the solution of partial differential equations on evolving domains. The approach uses a completely Eulerian description of the domain motion.
The physical domain is embedded in a triangulated computational domain and can overlap  the time-independent background mesh in an arbitrary way. The numerical method is based on finite difference discretizations of time derivatives and a standard geometrically unfitted finite element method with an additional stabilization term in the spatial domain.
The performance and analysis of the method rely on the fundamental extension result in Sobolev spaces for functions defined on bounded domains. This paper includes a complete stability and error analysis, which accounts for discretization errors resulting from finite difference and finite element approximations as well as for geometric errors coming from a possible approximate recovery of the physical domain. Several numerical examples illustrate the theory and demonstrate the practical efficiency of the method.
\end{abstract}
%
%
\maketitle

\section{Introduction}
Many mathematical models in physics, biology, chemistry and engineering involve partial differential equations (PDEs) posed on moving domains. Numerical simulations based on these models often face a challenge of building  discretizations, which handle accurately and efficiently both Lagrangian (displacement, material derivative) and Eulerian (temperature, concentration, local fluxes, etc.) quantities.  Several numerical approaches to accomplish this are known from the literature. For example, in the popular arbitrary Lagrangian--Eulerian approach~\cite{hirt1974arbitrary} one transforms the problem from a moving domain to a fixed reference domain through an artificial mapping and further applies meshing to the reference domain for the discretization purpose. The approach can be used both with spatial and space--time Galerkin formulations~\cite{masud1997space,tezduyar1992new}. The method allows for good resolution of the evolving domain boundary with a fitted mesh, but is known to be less practical in the case of larger deformations and unable to handle motions with topological changes. To overcome this deficiency, several methods based on a pure Eulerian description of the domain motion
have been developed over the past decades. The immersed boundary method~\cite{peskin1977numerical,peskin2002immersed} uses a fixed time-independent mesh to discretize both Eulerian and Lagrangian variables, linked by the Dirac delta function  smoothed over several layers of mesh cells.  Numerical methods that treat prorogating interfaces in a sharp way were developed more recently using the unfitted finite element technologies such as extended finite element methods~\cite{moes1999finite} and `cut' finite elements~\cite{cutFEM}. It is natural to combine unfitted finite elements with  space--time variational formulations of PDE in moving domain, and this line of research was taken in \cite{chessa2004arbitrary} (for 1D problem) and more recently expanded in \cite{LR_SINUM_2013,L_SISC_2015,hansbo2016cut}, including PDEs posed on evolving manifolds~\cite{grande2014eulerian,olshanskii2014eulerian,olshanskii2014error}.
These geometrically unfitted finite element methods are based on a fully Eulerian view point and exploit
the idea of using time-independent background finite element spaces. Therefore, these discretizations simplify the construction of numerical methods for domains that exhibit strong deformations or even topology changes.

Space--time Galerkin methods enjoy solid mathematical foundation (at least for scalar conservation laws) and both low and high order methods are easily formulated. On a practical side however, the reconstruction of space--time domains for the purpose of numerical integration is a difficult and possibly time consuming part.
To compensate for that, the space--time method in \cite{hansbo2016cut} introduces a variational crime by applying a quadrature in time approach to approximate space--time integrals. The resulting method does not require a reconstruction of space--time domains but only domain approximations at discrete time instances. However, to the best of the authors knowledge there is no theoretical bound for the varational crime commited by the quadrature (in time) of this method.
Further, arising linear systems in all previously mentioned space--time methods that need to be solved for are typically considerably larger than those of time stepping methods based on finite differences.

In this paper we abandon the use of a space--time variational framework  and opt for a more straightforward (and commonly used in steady domains) approach, where time discretization is based on finite difference approximations and the (unfitted) finite element method is used to accommodate spatial variations. The approach is based on the fundamental result of the existence of continuous extension operators (from a bounded domain to $\mathbb{R}^d,~d=2,3$) in Sobolev spaces. The result allows to identify the solution to the PDE with its smooth extension and further to design a finite element method, which solves at each discrete time instance for this extended solution in the computational domain. The acquired numerical extension allows one to apply finite differences to handle time derivatives in the physical domain. The remarkable feature of the method is that no explicit information about the extension is required, but a suitable numerical approximation to it becomes available through adding a simple stabilization term  to a standard unfitted finite element method formulation. This term acts in a narrow band containing the physical domain boundary.  This extension mechanism is different from the one in the classical fictitious domain methods~\cite{glowinski1999distributed}, where the PDE is extended from the physical to the computational domain.

The ideas similar to those elaborated in the present paper, were recently developed in \cite{olshanskii2017trace,LOX_ARXIV_2017} for the case of PDEs on moving surfaces, where we  combined stationary unfitted finite element discretizations (known as trace FEM) with time discretization schemes based on finite difference approximations. In those papers, the combination of both approaches has been enabled by adding a stabilization term which acts as a normal extension and facilitates the transition of information from the surface of one time step to the surface of the next time step.
To carry over this idea to volumetric domains here, we require a different mechanism that acts as a `smooth' (rather than normal) extension in the unfitted finite element method.
The idea of extending finite element solutions from an active part of the mesh at one time step to the active part of the next time step in order to apply a method of lines type approach is also found in \cite[Section 3.6.3]{schottstabilized}.
In \cite{schottstabilized} however only direct neighbors are involved in the extension which leads to time step restrictions obeying a geometrical CFL condition, $\Delta t \leq c h$. Further, in \cite{schottstabilized} the method has been applied without any theoretical error analysis.
In this work we propose and analyze a discretization without such a time restriction of CFL-type.
For the stabilization we consider a so-called ghost penalty method \cite{B10} and discuss three different version of it which share the same essential theoretical properties.
As a useful byproduct of the ghost penalty stabilization, the method possesses robustness w.r.t. the cut configuration not only in terms of error but also in terms of the conditioning of linear systems.

The remainder of the paper is organized as follows. In Section \ref{s:form} we formulate the model problem under investigation and propose a semi-discretization in time based on the idea of extension operators in Sobolev spaces in Section \ref{s:time}. The full discrete version of the method which includes a stabilization acting as a discrete extension operator is presented in Section \ref{s:FEM}. The a priori error analysis of the scheme is treated in Section \ref{s:Analysis}. We demonstrate the performance of the method based on numerical examples in Section \ref{s:Numerics} before we conclude with final remarks and open problems in Section \ref{s:Conclusions}.

\section{Mathematical problem} \label{s:form}
Consider a time-dependent domain $\Omega(t)\subset \rr^d,~d=2,3$ that is sufficiently regular for each $t\in [0,T]$, $T>0$, and evolves smoothly. More precisely, we shall assume the existence of a one-to-one continuous mapping
\begin{equation}\label{mapping}
\Psi(t)\,:\,\Omega_0\to\Omega(t)\quad\text{for each}~t\in[0,T],
\end{equation}
from the  reference domain $\Omega_0\subset \rr^d$. Later for the analysis, we need that $\partial\Omega_0$ is piecewise $C^2$ and Lipschitz and  $\Psi\in C^{m+1}([0,T]\times\overline{\Omega_0})$, where $m\ge1$ is the polynomial degree of our finite element space.
A polygonal background domain $\wOm$ is chosen such that $\Omega(t)$ together with its neighborhood is contained in $\wOm$ for all times $t \in [0,T]$.

One example of the suitable setup is given by a smooth motion and deformation  of the material volume $\Omega(t)$, e.g., volume of fluid. If $\bw:\Omega(t)\to\rr^d$ is the material velocity of the particles from $\Omega(t)$, then $\Psi(t)$ can be defined as the Lagrangian mapping from $\Om{}{0}=\Omega(0)$ to $\Omega(t)$, i.e. for $y\in\Om{}{0}$, $\Psi(t,y)$ solves the ODE system
\begin{equation}\label{Lagrange}
\Psi(0,y)=y,\quad  \frac{\partial \Psi(t,y)}{\partial t}=\bw(t,\Psi(t,y)),\quad t\in[0,T].
\end{equation}
%
%
The conservation of a scalar quantity $u$ with a diffusive flux in $\Omega(t)$ then leads to the equation
\begin{equation} \label{transport}
  \frac{\partial u}{\partial t} + \Div(u\bw) - \alpha \Delta u=0\quad\text{on}~~\Omega(t), ~~t\in (0,T],
\end{equation}
 with initial condition $u(\bx,0)=u_0(\bx)$ for $\bx \in \Omega(0)$. Here $\alpha>0$ is the constant diffusion coefficient. This is the model example of a parabolic PDE posed in a time-dependent domain that we use in this paper to formulate and analyze the finite element method. For simplicity we shall assume that the flux (which is only the diffusive flux) is zero on the boundary $\Gamma(t) := \partial \Omega(t)$,
\begin{equation} \label{bc}
\nabla u \cdot \bn = 0 \quad \text{on}~~\Gamma(t),~~t\in (0,T],
\end{equation}
where $\bn$ is the unit normal on $\Gamma(t)$.

These are the appropriate boundary conditions for a conserved quantity $u$. To see this, we apply Reynolds' transport theorem for moving domains:
$$
\frac{d}{dt} \int_{\Omega(t)} \!\!\!\!\!\! u ~ dx =  \int_{\Omega(t)} \frac{\partial}{\partial t} u ~ dx + \int_{\partial \Omega(t)} \!\!\! (\bw \cdot \bn) u ~ ds
=  \int_{\Omega(t)} \frac{\partial u}{\partial t} + \Div(u \bw) ~ dx
=  \int_{\Omega(t)} \!\!\!\!\!\! \alpha \Delta u ~ dx
=  \int_{\partial \Omega(t)} \!\!\!\!\!\! \alpha \nabla u \cdot \bn ~ ds.
$$
Later, we comment on the numerical treatment of other boundary conditions, see Remark~\ref{rem_bc}.

We emphasise that the proposed finite element method applies in a more general situation when one is only given $\Omega(t_n)$ or its approximation in some time instances $t_n\in[0,T]$ without any explicit information about $\Psi$. For the analysis, we need to assume that such mapping from the reference domain to the physical one at least exists and can be extended to a one-to-one mapping from a sufficiently large
neighborhood $\mathcal{O}(\Om{}{0})$ of $\Omega_0$ to $\mathcal{O}(\Om{}{}(t))$. This extended mapping, also denoted by $\Psi$, is assumed smooth, $\Psi \in C^{m+1}([0,T]\times\mathcal{O}(\Om{}{0}))$. 

For the analysis, we shall also need the notion of the space--time domain, where the problem \eqref{transport} is posed,
and its spatial neighborhood:
\[
\Qs = \bigcup\limits_{t \in (0,T)} \Omega(t) \times \{t\},\quad
\mathcal{O}(\Qs) = \bigcup\limits_{t \in (0,T)} \mathcal{O}(\Omega(t)) \times \{t\},\quad  \Qs\subset\mathcal{O}(\Qs)\subset \Bbb{R}^{d+1}.
\]


\section{Discretization in time} \label{s:time}
We first consider the discretization in time only. 
The goal of this paper is the study of a fully discrete method, but the treatment of the semi-discrete problem gives some insight and serves for the purpose of better exposition.

\subsection{Time discretization method}
For simplicity of notation, consider the uniform time step $\Delta t=T/N$, and let $t_n=n\Delta t$ and $I_n=[t_{n-1},t_n)$.
Denote by $u^n$ an approximation of $u(t_n)$, define $\Om{n}{} := \Omega(t_n)$, $\Gamma^{n} := \Gamma(t_n)$.

We define the $\delta$-neighborhood of $\Om{}{}(t)$ by
\begin{equation}
  \O{}(\Om{}{}(t)) := \{ \bx \in \rr^d\, : \dist(\bx,\Om{}{}(t)) \leq \delta \}.
\end{equation}
We require the neighborhood to be large enough
so that
\begin{equation}\label{ass1}
  \Om{n}{} \subset \O(\Om{n-1}{})\quad\text{for}~n=1,\dots,N.
\end{equation}
This can be assured by setting $\delta$ proportional to $\Delta t$ times the maximum normal velocity of $\Gamma$,
\begin{equation}\label{delta}
  \delta = c_\delta \winfn \Delta t, \text{ with } \winfn := \max_{t \in [0,T)} \norm{\bw \cdot \bn}_{L^{\infty}(\Gamma(t))} \text{ and } c_\delta > 1.
\end{equation}
In its turn, we also assume for each $n$ that $\O(\Om{n-1}{})\subset\mathcal{O}(\Om{n-1}{})$, a discretization independent ambient neighborhood, where the extended mapping $\Psi$ is defined. This is always the case for $\Delta t$ not too big.


In the time stepping method we combine the solution for $u^n$ on $\Om{n}{}$ with its extension on $\O{}(\Om{n}{})$ in every time step. This guarantees that $u^{n-1}$ is well-defined on $\Om{n}{}$, and we can approximate  the time derivative by a finite difference.   Thus, the implicit Euler method for
\eqref{transport} is
\begin{equation}\label{e:ImEuler}
\frac{u^n- \E{} u^{n-1}}{\Delta t}+\Div(u^n\bw) -\alpha\Delta u^n=0,\qquad\text{on}~~\Om{n}{}.
\end{equation}
Here, $\E{} : H^1(\Om{n-1}{}) \to  H^1(\mathcal{O}(\Om{n-1}{}))$ is a continuous extension operator. A suitable extension operator is defined in section \ref{s:extension} below based on the mapping $\Phi$ from \eqref{mapping}. Although $\E{}$ appears explicitly in \eqref{e:ImEuler}, it turns out that in the finite element setting (section~\ref{s:FEM}) a suitable extension can be defined implicitly and there is no need in any knowledge about $\Phi$.

\subsubsection*{Variational formulation in space}
We seek for $u^n \in H^1(\Om{n}{})$ such that for all $v \in H^1(\Om{n}{})$ there holds
\begin{equation} \label{e:contunfFEM1}
  \int_{\Om{n}{}} \frac{1}{\Delta t} u^n v\, dx + a^n(u^n,v) =
  \int_{\Om{n}{}} \frac{1}{\Delta t} \E{} u^{n-1} v \, dx.
\end{equation}
Here, $a^n(\cdot,\cdot)$ denotes the bilinear form for diffusion and convection where we use a  skew-symmetric formulation for the convection:
\begin{align}
  a^n(u,v) := &
  \int_{\Om{n}{}} \alpha \nabla u\cdot \nabla v\, dx
  + \label{e:an}
  \frac12 \int_{\Om{n}{}} (\bw\cdot\nabla u) \, v
  - (\bw\cdot\nabla v) \, u dx \\
 & + \frac12 \int_{\Om{n}{}} \Div(\bw) u v \, dx
 + \frac12\int_{\Gamma^n} (\bw \cdot  \bn) u v\, dx, \qquad u,v \in H^1(\Om{n}{}). \nonumber
\end{align}

We mention that the method has a straight-forward extension to higher order time integration, e.g. the BDF2 scheme, cf. Remark \ref{rem:bdf2} below. For ease of presentation we focus on the implicit Euler method first.

\subsubsection*{Unique solvability}
To guarantee unique solvability in every time step, we ask for coercivity of the left-hand side bilinear form in \eqref{e:contunfFEM1} with respect to $\Vert \cdot \Vert_{H^1(\Om{n}{})}$.

\begin{lemma}\label{lem:an}
For $u \in H^1(\Om{n}{})$ there holds
  \begin{equation}
  a^n(u,u) \geq \frac{\alpha}{2} \norm{\nabla u}_{L^2(\Om{n}{})}^2 - \xi \norm{u}_{L^2(\Om{n}{})}^2,
  \end{equation}
i.e. \eqref{e:contunfFEM1} is uniquely solvable if
  \begin{equation}\label{eq:xi}
  \Delta t < \xi^{-1} :=
  2\left( \Vert \Div(\bw) \Vert_{L^\infty(\Om{n}{})} + {c_{\Omega}^2  \Vert \bw \cdot \bn \Vert_{L^\infty(\Om{n}{})}^2}/{(4 \alpha) + \alpha} \right)^{-1}
\end{equation}
where $c_\Omega$ is the constant of the multiplicative trace inequality $ \Vert u \Vert_{L^2(\Gamma^n)}^2 \leq c_{\Omega} \norm{u}_{L^2(\Om{n}{})} \norm{u}_{H^1(\Om{n}{})} $.
\end{lemma}
\begin{proof}
  Due to
\begin{equation*}
 a^n(u,u) \geq \alpha \Vert \nabla u \Vert_{L^2(\Om{n}{})}^2 - \frac12 \Vert \Div(\bw) \Vert_{L^\infty(\Om{n}{})} \Vert u \Vert_{L^2(\Om{n}{})}^2 - \frac12 \Vert \bw \cdot \bn \Vert_{L^\infty(\Gamma^n)} \Vert u \Vert_{L^2(\Gamma^n)}^2,
\end{equation*}
the multiplicative trace inequality and Young's inequality we have
\begin{align*}
  \Vert \bw \cdot \bn \Vert_{L^\infty(\Gamma^n)} \Vert u \Vert_{L^2(\Gamma^n)}^2  &\leq c_{\Omega} \Vert \bw \cdot \bn \Vert_{L^\infty(\Gamma^n)} \norm{u}_{L^2(\Om{n}{})} \norm{u}_{H^1(\Om{n}{})} \\
  &\leq {c_\Omega^2 \Vert \bw \cdot \bn \Vert_{L^\infty(\Gamma^n)}^2}/{(4\alpha)} \  \norm{u}_{L^2(\Om{n}{})}^2 + \alpha \norm{u}_{H^1(\Om{n}{})}^2 \\
  &= \left( {c_\Omega^2 \Vert \bw \cdot \bn \Vert_{L^\infty(\Gamma^n)}^2}/{(4\alpha)} + \alpha \right) \norm{u}_{L^2(\Om{n}{})}^2 + \alpha \norm{\nabla u}_{L^2(\Om{n}{})}^2
\end{align*}
which yields
\begin{equation*}
  \Delta t^{-1} \Vert u \Vert_{L^2(\Om{n}{})}^2 + a^n(u,u) \geq \left( \Delta t ^{-1} -     \xi
  \right) \Vert u \Vert_{L^2(\Om{n}{})}^2 + \frac{\alpha}{2} \Vert \nabla u \Vert_{L^2(\Om{n}{})}^2 .
\end{equation*}
\end{proof}

\begin{remark}[Dirichlet boundary conditions]\rm
  If we consider $u = g_D$ for a given function $g_D \in H^{\frac12}(\Gamma(t))$ as boundary condition that is implemented through the Sobolev spaces in the variational formulation instead of \eqref{bc}, the last integral involving the boundary integral in \eqref{e:an} would vanish and the condition in Lemma \ref{lem:an} would simplify to $\Delta t <2 \norm{\Div(\bw)}_{L^\infty(\Om{n}{})}^{-1}$.
\end{remark}


\subsection{Stability of the semi-discrete method} \label{s:stab:semi-disc}
In this section we show a numerical stability bound for $u^n$.
We use the following abbreviations in norms and scalar products for functions $u,v$ in a domain $G$:
$\innerprod{u}{v}_{G} := \innerprod{u}{v}_{L^2(G)}$.


\subsubsection{Extension operator} \label{s:extension}
To define an extension operator from the time dependent domain to its neighborhood, we first assume such an extension on the initial domain and define a corresponding extension for $t>0$ by transformation.

Let $m\ge1$ be a fixed integer, since the boundary of $\Om{}{0}$ is piecewise smooth and Lipschitz, there is a continuous linear extension operator $\widehat{\E{0}} : L^2(\Om{}{0}) \to  L^2(\mathcal{O}(\Om{}{0}))$,  ($\widehat{\E{0}} u = u$ in $\Om{}{0}$), with the following  properties~\cite[Section~VI.3.1]{stein2016singular}:
\begin{equation}\label{ExtBound0}
  \Vert \widehat{\E{0}} u \Vert_{W^{k,p}(\mathcal{O}(\Om{}{0}))} \leq \hat C_{\Om{}{0}} \Vert u \Vert_{W^{k,p}(\Om{}{0})},\quad \text{for}~u \in W^{k,p}(\Om{}{0}),~~k=0,\dots,m+1,~~1\le p\le\infty.
\end{equation}

Note that we can always decompose $v$ from  $L^2(\Om{}{0})$ as $v=u+|\Om{}{0}|^{-1}\int_{\Om{}{0}}v\,dx$ and define the extension  $\E{0}v=\widehat{\E{0}}u+|\Om{}{0}|^{-1}\int_{\Om{}{0}}v\,dx$, then the updated extension operator satisfies same bounds as in \eqref{ExtBound0} and thanks to the Poincar\'{e} inequality
\begin{equation}\label{ExtBound}
\begin{aligned}
  \Vert \nabla\E{0} v \Vert_{\mathcal{O}(\Om{}{0})} \leq C_{\Om{}{0}} \Vert \nabla v \Vert_{\Om{}{0}}, ~~
  \quad \text{for}~u \in H^{1}(\Om{}{0}).
\end{aligned}
\end{equation}
{
We shall need the following commutation property of the extension operator and time derivative.
\begin{lemma}\label{L_exch}
Let $Q_0:=\Om{}{0} \times (0,T)$ and $\mathcal{O}(Q_0):=\mathcal{O}(\Om{}{0}) \times (0,T)$. For $v\in L^2(Q_0)$ such that $v_t\in L^2(Q_0)$, it holds $ \left(\E{0} v\right)_t\in L^2(\mathcal{O}(Q_0))$ and
\begin{equation*}
 \left(\E{0} v\right)_t=\E{0} v_t\quad\text{in}~~\mathcal{O}(Q_0).
\end{equation*}
\end{lemma}
\begin{proof}
The result follows from the linearity and continuity of $\E{0}$ and a density argument. For completeness we included the elementary proof in the appendix.
\end{proof}
}
The mapping $\Psi(t)$ from \eqref{mapping} defines a diffeomorphism at every time $t$ between $\Om{}{0}$ and $\Om{}{}(t)$ and $\mathcal{O}(\Om{}{0})$ and $\mathcal{O}(\Om{}{}(t))$, respectively. Using this mapping we define the extension
\begin{equation} \label{e:extensiont_cont}
\E{} u(t) := (\E{0} (u \circ \Psi(t))) \circ \Psi^{-1}(t),\quad \text{for each}~t \in [0,T].
\end{equation}
Note that $\E{}u$ can be also seen as an extension of $u$ from $\Qs$ to $\mathcal{O}(\Qs)$.
Further we shall assume certain regularity of solution to \eqref{transport} in terms of space--time anisotropic spaces
\footnote{The definition differs from that of Bochner-type spaces in time-dependent domains found in \cite{alphonse2015abstract}, but suffices for what follows.}
\[
L^\infty(0,T;H^{k}(\Om{}{}(t))):=\{v\in L^2(Q)\,:\, v \circ \Psi(t)\in H^k(\Om{}{0})~\text{for a.e.}~t\in(0,T)~~\text{and}~~\operatornamewithlimits{ess\,sup}\limits_{t\in(0,T)}\|v\circ \Psi(t)\|_{H^k(\Om{}{0})}<\infty\},
\]
$k=0,\dots,m+1.$ Thanks to the smoothness of $\Psi$, it holds
\[
\operatornamewithlimits{ess\,sup}\limits_{t\in(0,T)}\|v(t)\|_{H^k(\Om{}{}(t))}<\infty \quad\text{for}~v\in L^\infty(0,T;H^{k}(\Om{}{}(t))).
\]
For $v\in L^2(Q)$, $v_t$ denotes weak  partial derivative w.r.t. the time variable, if it exists as an element of $L^2(Q)$.


 We need the following properties of the extension.
\begin{lemma} \label{lem:extt}
  For $u\in L^\infty(0,T;H^{m+1}(\Om{}{}(t)))\cap W^{2,\infty}(\Qs)$ there holds
\begin{subequations}\label{u_bound_a}
\begin{align}
\|\E{}u\|_{H^{k}(\O(\Om{}{}(t)))}&\le c_{L\ref{lem:extt}a} \|u\|_{H^{k}(\Om{}{}(t))}, \quad {\small k=0,\dots,m+1},\label{u_bound_a1}\\
  \Vert \nabla (\E{} u) \Vert_{\O(\Om{n}{}(t))} & \leq c_{L\ref{lem:extt}b} \Vert \nabla u \Vert_{\Om{n}{}(t)},\label{u_bound_a2}\\
  \|\E{}u\|_{W^{2,\infty}(\O(\Qs))}&\le c_{L\ref{lem:extt}c} \|u\|_{W^{2,\infty}(\Qs)}, \label{u_bound_a3}
\end{align}
\end{subequations}
 with constants $c_{L\ref{lem:extt}a}$, $c_{L\ref{lem:extt}b}$,  $c_{L\ref{lem:extt}d}$ depending only on $\Psi$. Furthermore, for
$u\in L^\infty(0,T;H^{m+1}(\Om{}{}(t)))$ such that $u_t\in L^\infty(0,T;H^{m}(\Om{}{}(t)))$ it holds
\begin{equation}\label{u_bound_b}
\|(\E{}u)_t\|_{H^{m}(\O(\Om{}{}(t)))}\le c_{L\ref{lem:extt}e} (\|u\|_{H^{m+1}(\Om{}{}(t))}+\|u_t\|_{H^{m}(\Om{}{}(t))}),
\end{equation}
where $c_{L\ref{lem:extt}e}$ depends only on $\Psi$.
\end{lemma}
\begin{proof}
  The proof of \eqref{u_bound_a1}--\eqref{u_bound_a3} follows by the standard arguments based on the  transformation formulas \eqref{e:extensiont_cont}, the differentiation chain rule,  the smoothness assumption for the mapping: $\Psi\in C^{m+1}([0,T]\times\overline{\Omega_0})$,   \eqref{ExtBound0}--\eqref{ExtBound}, Lemma~\ref{L_exch} and $\O(\Om{}{}(t)) \subset \mathcal{O}(\Om{}{}(t))$, $\O(Q) \subset \mathcal{O}(Q)$. We draft the proof of \eqref{u_bound_b}, since it requires a little bit more computations. By the definition of the extension, one gets $(\E{}u)_t=(\E{0} (u \circ \Psi(t)))_t \circ \Psi^{-1}(t)+
  (\Psi^{-1})_t\nabla(\E{0} (u \circ \Psi(t))) \circ \Psi^{-1}(t)$. Thanks to the smoothness of $\Psi$ it holds
  \[
     \|(\E{}u)_t\|_{H^{m}(\O(\Om{}{}(t)))}\le c(\|\nabla(\E{0} (u \circ \Psi(t)))\|_{H^{m}(\O(\Om{}{0}))}+\|(\E{0} (u \circ \Psi(t)))_t\|_{H^{m}(\O(\Om{}{0}))}),
  \]
Using the result of Lemma~\ref{L_exch} and the smoothness of the mapping again, we proceed with
   \[
  \begin{split}
     \|(\E{}u)_t\|_{H^{m}(\O(\Om{}{}(t)))}&\le  c(\|\E{0} (u \circ \Psi(t))\|_{H^{m+1}(\O(\Om{}{0}))}+\|\E{0}( (u \circ \Psi(t))_t)\|_{H^{m}(\O(\Om{}{0}))})\\
             & \le c(\|u \circ \Psi(t)\|_{H^{m+1}(\Om{}{0})}+\|(u \circ \Psi(t))_t\|_{H^{m}(\Om{}{0})})\\
              & \le c(\|u \|_{H^{m+1}(\Om{}{}(t))}+\|\Psi_t(t) \nabla(u \circ \Psi(t))\|_{H^{m}(\Om{}{0})})
              +\|u_t \circ \Psi(t)\|_{H^{m}(\Om{}{0})})\\
               & \le c(\|u \|_{H^{m+1}(\Om{}{}(t))}+\|u_t \circ \Psi(t)\|_{H^{m}(\Om{}{}(t))}).
  \end{split}
  \]
\end{proof}
\smallskip

\subsubsection{Stability} Now we are ready to show the stability of the semi-discrete method. To avoid extra technical complication for piecewise smooth boundary, we shall assume that $\partial\Omega_0$ is $C^2$.

\begin{lemma}\label{l_est1} For $u\in H^1(\Om{n}{})$, $n=1,\dots, N$, there holds for any $\ep > 0$
\begin{equation}\label{est1}
  \| \E{} u \|_{\O{}(\Om{n}{})}^2 \le (1+(1+{\ep}^{-1})\, \delta \, c_{L\ref{l_est1}a})  \|u\|_{\Om{n}{}}^2 + \delta \, c_{L\ref{l_est1}b} \ep \| \nabla u \|_{\Om{n}{}}^2
\end{equation}
for constants $c_{L\ref{l_est1}a}$, $c_{L\ref{l_est1}b}$ independent of $\Delta t$, $n$ and $u$, once $\Delta t$ is sufficiently small.
\end{lemma}
\begin{proof} The proof largely follows the arguments found in \cite[Theorem 1.5.1.10]{grisvard2011elliptic} and \cite[Lemma 4.10]{elliott2012finite}.
Let us define the strip $\S(\Om{n}{}) = \O(\Om{n}{}) \setminus \Om{n}{}$. Since
\[
 \| \E{} u \|_{\O{}(\Om{n}{})}^2= \| \E{} u \|_{\S(\Om{n}{})}^2+\| u \|_{\Om{n}{}}^2,
\]
we need to prove
 \begin{equation}\label{est1a}
  \| \E{} u \|_{\S(\Om{n}{})}^2 \le (1+{\ep}^{-1})\, \delta \, c_{L\ref{l_est1}a}  \|u\|_{\Om{n}{}}^2 + \delta \, c_{L\ref{l_est1}b} \ep \| \nabla u \|_{\Om{n}{}}^2.
\end{equation}
We define a  function $\hat \phi$ such that $\hat\phi$ is the signed distance function to $\Gamma^n$ in $S_\delta(\Om{n}{})$.
We have $\|\hat\phi\|_{C^2(\S(\Om{n}{}))}\le c_n$, for $\delta\le\Delta t$, with $c_n$ and $\Delta t$ depending only on the curvature of $\Gamma^n$ \cite{federer1959curvature}, and hence
\begin{equation}\label{phi}
\sup_{n=1,\dots,N}\|\hat\phi\|_{C^2(\S(\Om{n}{}))}\le c,
\end{equation}
with finite $c$ depending on $\Omega_0$ and $\Psi$ for sufficiently small $\Delta t$.  We set $\phi:=\E{}\hat\phi$ to be the extension of $\phi$ to $\O(\Om{n}{})$.  Let $\Gamma_r=\{x\in \S(\Om{n}{})\,:\,\phi(x)=r\}$, $r\in[0,\delta]$, $\bn_r$ the outward normal vector and $\Omega_r$ the $r$-neighborhood of $\Omega^n$, i.e., $\Gamma_r=\partial\Omega_r$. Applying the Green's formula, one shows the identity,
\[
\int_{\Gamma_r} (\E{}u)^2 \bn_r\cdot\nabla\phi\, ds= 2\int_{\Omega_r} (\E{}u)\nabla(\E{}u)\cdot\nabla\phi\, ds+
\int_{\Omega_r} (\E{}u)^2\Delta\phi\, ds.
\]
Using $\bn_r\cdot\nabla\phi=1$, Cauchy-Schwarz and Young's inequalities yields
\begin{equation*}
\begin{split}
\int_{\Gamma_r} (\E{}u)^2  ds&\le \|\phi\|_{C^2(\O(\Om{n}{}))}\left(\int_{\Omega_r} |\E{}u||\nabla(\E{}u)|ds+
\int_{\Omega_r} (\E{}u)^2 ds\right)\\ &\le \|\phi\|_{W^{2,\infty}(\O(\Om{n}{}))}\left((1+{\ep}^{-1})\,\|\E{}u\|_{\Omega_r}^2+
\frac\ep4 \| \nabla \E{}u \|_{\Omega_r}^2\right) .
\end{split}
\end{equation*}
Thanks to Lemma~\ref{lem:extt} and \eqref{phi} it holds $\|\phi\|_{W^{2,\infty}(\O(\Om{n}{}))}\le c$, with a constant $c$ independent of $n$. This, the embedding $\Omega_r\subset\O(\Om{n}{})$ and Lemma~\ref{lem:extt} again imply
 \begin{equation}\label{est1b}
\int_{\Gamma_r} (\E{}u)^2  ds \le C\left(c_{L\ref{lem:extt}a}^n(1+{\ep}^{-1})\,\|u\|_{\Omega^n}^2+
c_{L\ref{lem:extt}b}^n\frac\ep4 \| \nabla u \|_{\Omega^n}^2\right).
\end{equation}
By the co-area formula and using $|\nabla\phi|=1$ in $\S(\Om{n}{})$  we have
\[
\| \E{} u \|_{\S(\Om{n}{})}^2 =\int_0^\delta\int_{\Gamma_r} (\E{}u)^2 ds\,dr.
\]
Therefore integrating \eqref{est1b} over $r\in(0,\delta)$ yields \eqref{est1a}.
\end{proof}
\medskip
In the next lemma we show \emph{numerical stability} of the semi-discrete scheme.
\begin{lemma}\label{lem:contunfFEM1}
  For $\Delta t$ sufficiently small and $\{u^n\}_{n=1,\dots,N}$ the solution of \eqref{e:contunfFEM1} with initial data $u^0 \in L^2(\Om{}{0})$ there holds
  \begin{equation}\label{stab_semi}
\norm{u^k}_{\Om{k}{}}^2 + \Delta t  \sum_{n=1}^k \alpha/2 \ \norm{\nabla u^n}_{\Om{n}{}}^2 \le \exp(c_{L\ref{lem:contunfFEM1}} t_k)\|u^0\|_{\Om{}{0}}^2,\quad \text{for}~k=0,\dots,N,
\end{equation}
for a constant $c_{L\ref{lem:contunfFEM1}}$ that is independent of $\Delta t$ and $k$.
\end{lemma}
\begin{proof}
  We test \eqref{e:contunfFEM1} with $2 u^n$ and get
  \begin{align}
    \norm{u^n}_{\Om{n}{}}^2 + \norm{u^n- \E{} u^{n-1}}_{\Om{n}{}}^2 + 2 \Delta t a^n(u^n,u^n) = \norm{\E{}u^{n-1}}_{\Om{n}{}}^2.
  \end{align}
  With Lemma \ref{lem:an}, Lemma \ref{l_est1} and $\delta$ from \eqref{delta} we get
  \begin{align}
    (1- & 2\xi \Delta t) \norm{u^n}_{\Om{n}{}}^2 + \Delta t \alpha  \norm{\nabla u^n}_{\Om{n}{}}^2 \leq \norm{\E{}u^{n-1}}_{\Om{n}{}}^2 \leq \norm{\E{}u^{n-1}}_{\O(\Om{n-1}{})}^2 \nonumber \\
     & \leq
 (1+(1+{\ep}^{-1})\, c_\delta \winfn c_{L\ref{l_est1}a} \Delta t)  \|u^{n-1}\|_{\Om{n-1}{}}^2 + c_\delta \winfn \Delta t \, c_{L\ref{l_est1}b} \ep \| \nabla u^{n-1} \|_{\Om{n-1}{}}^2.
  \end{align}
  We choose $\ep = {\alpha} (2 c_{L\ref{l_est1}b} c_\delta \winfn)^{-1} $ and obtain
  \begin{align}
    (1- &2 \xi \Delta t ) \norm{u^n}_{\Om{n}{}}^2 + \Delta t \alpha  \norm{\nabla u^n}_{\Om{n}{}}^2 \leq \norm{\E{}u^{n-1}}_{\Om{n}{}}^2 \leq \norm{\E{}u^{n-1}}_{\O(\Om{n-1}{})}^2 \nonumber \\
     & \leq
 (1+\bar{c} \Delta t)  \|u^{n-1}\|_{\Om{n-1}{}}^2 + \frac{\alpha}{2} \Delta t \, \| \nabla u^{n-1} \|_{\Om{n-1}{}}^2
  \end{align}
  with $\bar{c} = (1+{\ep}^{-1})\, c_\delta \winfn c_{L\ref{l_est1}a}$ independent of $\Delta t$.
  Summing up  over $n=1,\dots,k,~k\leq N$ yields
  \begin{align}
    (1- & 2\xi \Delta t ) \norm{u^k}_{\Om{k}{}}^2 + \Delta t \alpha/2  \sum_{n=1}^k  \norm{\nabla u^n}_{\Om{n}{}}^2 \leq \norm{u^0}_{\Om{}{0}}^2 + (2 \xi+\bar{c}) \Delta t \sum_{n=0}^{k-1} \| u^n \|_{\Om{n}{}}^2.
  \end{align}
  Assuming that $\Delta t$ is sufficiently small so that $\xi \Delta t < \frac14$, we apply the discrete Gronwall inequality which yields the results with $c_{L\ref{lem:contunfFEM1}} = \bar{c} + 2 \xi$.
\end{proof}

Now we turn to the fully discrete case. Besides the technical difficulties of passing from differential equations to algebraic and finite element functional spaces, we need to handle the situation, when the smooth domain $\Om{n}{}$ is approximated by a set of piecewise smooth $\Om{n}{h}$, $n=0,\dots,N$.

\section{Discretization in space and time} \label{s:FEM}
\subsection{Meshes and finite element spaces} \label{s:meshes}
Assume a  family of consistent  subdivisions of $\wOm$ into a quasi-uniform triangulation $\{\T_h\}_{h>0}$ consisting of simplexes with a characteristic mesh size $h$.
$V_h$ denotes the time-independent finite element space,
\begin{equation} \label{eq:Vh}
V_h:=\{v_h\in C(\Omega)\,:\, v_h|_S\in P_m(S), \forall S\in \mathcal{T}_h\},\quad m\ge1,
\end{equation}
where $P_m(S)$ is the space of polynomials of at most degree $m$ on $S$.
The domains $\Om{n}{},~n=1,\dots,N$, are approximated by discrete approximations $\Om{n}{h},~n=1,\dots,N$, e.g. using an approximated level set function, cf. Section \ref{sec:geomapprox}.

In the full discrete method, we combine the solution and the extension step that we have seen in the semi-discrete method by one stabilized solve on a discretely extended domain.
In every time step, we extend the domain $\Om{n}{h}$ by a layer of thickness $\delta_h$ where we choose $\delta_h$ so that $\Om{n+1}{h}$ is a subset of the extended domain to $\Om{n}{h}$, $  \delta_h > \winfn ~\Delta t$.
To this end, we define the \emph{active mesh}, the set of all elements that have some part in this extended domain,
\begin{equation}
  \Td{n} := \{ S \in \T_h: \dist(\bx,\Om{n}{h}) \leq \delta_h \text{ for some } \bx \in S\}, \quad
  \Odt{n} := \{ \bx \in S: S \in \Td{n} \}.
\end{equation}
See the left sketch in Figure \ref{fig:discretedomains} for an example.
\begin{figure}
  \begin{center}
    \includegraphics[width=0.66\textwidth]{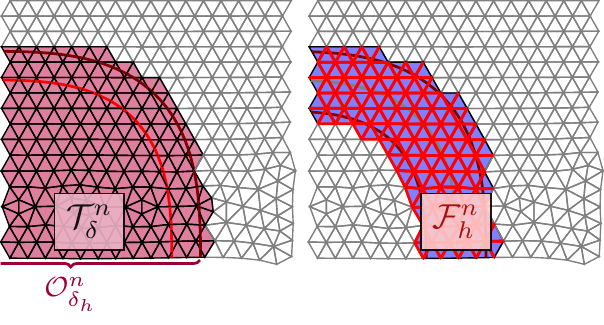}
  \end{center}
  \vspace*{-0.5cm}
  \caption{Sketch of active part of the mesh (left) and active facets (right) that are utilized in the stabilization.}
  \vspace*{-0.2cm}
  \label{fig:discretedomains}
\end{figure}
On these active meshes, we define the finite element spaces
\begin{equation}\label{defV}
V_h^n:= \{v \in C(\Odt{n})\,:\, v\in P_m(S), \forall~ S \in \Td{n} \},\quad m\ge1.
\end{equation}
These spaces are the restrictions of the time-independent bulk space $V_h$ on all simplices from $\Td{n}$.

\subsection{Variational formulation} \label{s:varform}
The numerical method is based on the semi-discrete formulation \eqref{e:contunfFEM1}. Instead of applying an extension step separately we add a stabilization term $s_h^n(\cdot,\cdot)$ that realizes a discrete version of the extension. It reads as:
For a given $u_h^0 \in V_h^0$ find  $u_h^n\in V_h^n$, $n=1,\dots,N$, satisfying
\begin{equation}\label{e:unfFEM1}
\int_{\Om{n}{h}} \frac{u_h^n-u_h^{n-1}}{\Delta t} v_h + \ahn(u^n_h, v_h)+ \gamma_s s_h^n(u_h^n,v_h) =0 \quad \text{ for all } v_h\in V_h^n;
\end{equation}
$s_h^n(\cdot,\cdot)$ is a stabilization bilinear form that is yet to be defined below, cf. Section \ref{s:gp}, $\gamma_s=\gamma_s(h,\delta_h)$ is a stabilization parameter, cf. Section \ref{s:gammas}, and
\begin{align}
  \ahn(u_h,v_h) := &
                      \int_{\Om{n}{h}} \alpha \nabla u_h\cdot \nabla v_h\, dx
  \label{e:anh}
  +
  \frac12 \int_{\Om{n}{h}} \left( (\bw^e\cdot\nabla u_h) \, v_h
  - (\bw^e\cdot\nabla v_h) \, u_h \right) \, dx \\
 & + \frac12 \int_{\Om{n}{h}} \Div(\bw^e) u_h v_h \, dx
 + \frac12\int_{\Gamma_h^n} (\bw^e \cdot  \bn) u_h v_h \, dx, \qquad u_h,v_h \in H^1(\Om{n}{h}). \nonumber
\end{align}
Here, $\we$ is a suitable smooth extension of the velocity field $\bw$ that is only defined on $\Om{n}{} \neq \Om{n}{h}$.

\subsection{Stabilization bilinear forms}\label{s:gp}
The stabilizing bilinear form $s_h^n(\cdot,\cdot)$ has multiple purposes. First, it should stabilize the solution of the problem \eqref{e:unfFEM1} due to irregular cuts. Secondly, it is responsible for the implicit definition of an extension to $\Odt{n}$. Thirdly, it provides condition number bounds that are independent of the cut position, cf. Remark \ref{rem:linalg} below.

We present three possible choices in the next subsections on how to choose the stabilization term. They are all in a very similar flavour and share the crucial properties needed in this method.
All these stabilizations have in common that they add stabilization terms on facets in the region of
the boundary $\Gamma_h^n = \partial \Om{n}{h}$. To this end, we define the elements that are in the boundary strip:
\begin{equation}
\TS{n} :=\{ S \in \Td{n}\, : \dist(\bx,\Gamma_h^n) \leq \delta_h \text{ for some } \bx \in S\} .
\end{equation}
We notice that the boundary strip includes cut elements, but possibly also some elements that are completely inside or outside of $\Om{n}{h}$. We define the set of facets between elements in $\Td{n}$ and $\TS{n}$:
\begin{equation} \label{e:deffhk}
  \Fh^n := \{ \overline{T_1} \cap \overline{T_2}\, : T_1 \in \Td{n},~ T_2 \in \TS{n}, T_1\neq T_2, \meas_{d-1} (\overline{T_1} \cap \overline{T_2}) > 0 \}.
\end{equation}
See the right sketch in Figure \ref{fig:discretedomains} for an example.
\subsubsection{``Direct'' version of the ghost penalty method} \label{s:gpdirect}
 For $F \in \Fh^n$ let $\omega_F$ be the facet patch, i.e. $\omega_F = T_1 \cup T_2$ for $T_1$ and $T_2$ as in the definition \eqref{e:deffhk}. We define for $u,v \in V_h^n$
\begin{equation}
  s_h^{n,\text{dir}}(u,v) := \sum_{F \in \Fh^n} s_{h,F}^{n,\text{dir}}(u,v) \quad \text{with} \quad s_{h,F}^{n,\text{dir}}(u,v):= \frac{1}{h^2} \int_{\omega_F} (u_1 - u_2) (v_1 - v_2) dx,
\end{equation}
where $u_1 = \mathcal{E}^P  u|_{T_1}$, $u_2 = \mathcal{E}^P  u|_{T_2}$ (and similarly for $v_1$, $v_2$) where $\mathcal{E}^P: P_m(S) \rightarrow P_m(\rr^d)$ is the canonical extension of a polynomial to $\rr^d$. This version of the ghost penalty stabilization has been proposed for the first time -- to the best of our knowledge -- in \cite{preussmaster}. Compared to other version of the ghost penalty method, cf. the sections below, this version has the advantage that an implementation of the bilinear form is only implicitly (through the extension $\mathcal{E}^P$) depending on the polynomial degree $m$.

For the analysis, we also define $s_h^{n,\text{dir}}(u,v)$ for arbitrary functions $u,v  \in L^2(\Odt{n})$. In this case, we set $u_1 = \mathcal{E}^P \Pi_{T_1} u|_{T_1}$, $u_2 = \mathcal{E}^P \Pi_{T_2} u|_{T_2}$
where $\Pi_{T_i}$ is the $L^2(T_i)$-projection into $P_m(T_i),~i=1,2$. We notice that for $v \in V_h^n$, $\Pi_{T_i} v|_{T_i} = v|_{T_i}$.
\subsubsection{LPS-type version of the ghost penalty method} \label{s:gplps}
The second version has been proposed for the first time in the original paper \cite{B10}. We call it the LPS (local projection stabilization)-type version of the ghost penalty method. It also formulates integrals over the facet patches $\omega_F$. But now, the deviation from a polynomial on the patch is penalized:
\begin{align} \label{e:lps}
  s_h^{n,\text{LPS}}(u,v) & :=\!\! \sum_{F \in \Fh^n} s_{h,F}^{n,\text{LPS}}(u,v) \\ \text{with} \ s_{h,F}^{n,\text{LPS}}(u,v) & :=
  \frac{1}{h^2} \int_{\omega_F} \!\! (u - \Pi_{\omega_F} u) (v - \Pi_{\omega_F} v)dx = \frac{1}{h^2} \int_{\omega_F} \!\! (u - \Pi_{\omega_F} u) v dx, \quad u,v \in L^2(\Odt{n}), \nonumber
\end{align}
where $\Pi_{\omega_F}: L^2(\omega_F) \rightarrow P_m(\omega_F)$ is the $L^2$ projection into the space of polynomials up to degree $m$ on $\omega_F$. We note that the last equality in \eqref{e:lps} holds due to the orthogonality of the $L^2$ projection.

\subsubsection{Derivative jump version of the ghost penalty method} \label{s:gpdjump}
The most well-known version of the ghost penalty stabilization is the following based on penalizing jumps in the (higher order) derivatives across facets, cf. e.g. \cite{massing2014stabilized,cutFEM,burman2014fictitious,schott2014new}:
\begin{equation}
  s_h^{n,\text{djmp}}(u,v) := \!\! \sum_{F \in \Fh^n} s_{h,F}^{n,\text{djmp}}(u,v) \ \text{with} \ s_{h,F}^{n,\text{djmp}}(u,v) := \!\! \sum_{k=0}^m \frac{h^{2k-1}}{k!^2} \! \int_{F} \jump{ \partial_{\bn_F}^k u } \jump{ \partial_{\bn_F}^k v } dx, \ u,v \in H^{m+1}(\Td{n}),
\end{equation}
where $\partial_{\bn_F}^k$ is the $k$-th directional derivative in the direction of the facet normal $\bn_F$. We note that the summand to $k=0$ is not required in an implementation due to the continuity of functions in $V_h^n$.

\subsection{Stabilization parameter $\gamma_s$} \label{s:gammas}
The stiffness between the unknowns on two elements $T \in \Td{n}$ and $T' \in \TS{n}$ that is induced by the stabilization bilinear form $s_h^n(\cdot,\cdot)$ depends reciprocally on the distance between the elements $T$ and $T'$ measured in terms of the number of facets that need to be crossed to walk through the mesh from $T$ to $T'$. This number  $K$ depends on the anisotropy between spatial and temporal discretization,
\begin{equation}
   K \leq c_{K,1} (1 + \delta_h/h), \qquad \delta_h \leq c_{K,2} \Delta t, \qquad \text{ with } c_{K,1}, c_{K,2} \text{ independent of } \Delta t \text{ and } h.
\end{equation}
Below in the analysis we will see that we require $\gamma_s \geq c_{K,3} K$ (with $c_{K,3}$ independent of $\Delta t$ and $h$) to compensate for the weakening of the stabilization for extension strips of increasing size. Hence, we choose
\begin{equation} \label{e:cgamma}
  \gamma_s = \gamma_s(h,\delta_h) = c_{\gamma} \, K \ \text{ with } \ c_\gamma>0 \text{ independent of } \Delta t \text{ and } h.
\end{equation}

\subsection{Additional remarks}
\begin{remark}[Unique solvability]\rm
Similarly to Lemma \ref{lem:an} we can easily check that
  \begin{align} \label{e:anhbound}
  \ahn(u_h,u_h) & \geq \frac{\alpha}{2} \norm{\nabla u_h}_{\Om{n}{h}}^2 - {\xi_h} \norm{u_h}_{\Om{n}{h}}^2,  \\
\label{eq:xih}
\text{ if } \qquad  \Delta t & <  \xi_h^{-1} := 2
  \left( \Vert \Div(\bw^e) \Vert_{L^\infty(\Om{n}{h})} + \alpha + {c_{\Omega_h}^2  \Vert \bw^e \cdot \bn \Vert_{L^\infty(\Om{n}{h})}^2}/{4 \alpha} \right)^{-1}.
\end{align}
Hence the left hand side bilinear form in \eqref{e:unfFEM1} is coercive on $V_h^n$ w.r.t. the norm
\begin{equation} \label{e:enorm}
\enorm{v}_n := \left(\frac{\alpha}{2} \norm{\nabla v}_{\Om{n}{h}}^2 + \norm{v}_{\Om{n}{h}}^2 + \gamma_s s_h^n(v,v) \right)^{1/2}.
\end{equation}
That this is actually a norm on $V_h^n$ is due to the properties of $s_h^n(\cdot,\cdot)$ treated below, see Lemma \ref{lem:gp}.
\end{remark}

\begin{remark}[Implementation of Dirichlet boundary conditions] \label{rem_bc}\rm
  We comment on the use and implementation of Dirichlet boundary conditions. If we consider Dirichlet boundary conditions $u = g_D$, we suggest to use (the unfitted version of) Nitsche's method for its implementation. In this case, the following bi- and linear forms have to be added to the discretization in \eqref{e:unfFEM1}:
  \begin{align}
    n_h^n(u_h,v_h) & := \int_{\Gamma_h^n}\{ (- \nabla u_h \cdot \bn) v_h + (- \nabla v_h \cdot \bn) u_h + \lambda_h u_h v_h \}\, d s,  \\
    g_h^n(v_h) & := \int_{\Gamma_h^n} g_D^e \, (- \nabla v_h \cdot \bn + \lambda_h v_h + \frac12 (\bw^e \cdot \bn) v_h ) \, d s,
  \end{align}
where $g_D^e$ is a suitable extension of $g_D$ from $\Gamma_n$ to $\Gamma_n^h$. Coercivity of the arising left hand side  bilinear form is then obtained for sufficiently large $\lambda_h$ and $\gamma_s$, cf., e.g., \cite{burman2012fictitious}.
\end{remark}

\section{Analysis of the fully discrete method} \label{s:Analysis}
In this section we carry out the numerical analysis of the fully discrete method.

\subsection{Preliminaries and notation} \label{s:prelim}
In order to reduce the repeated use of generic but unspecified constants,  further in the paper we write $x\lesssim y$ to state that the inequality  $x\le c y$ holds for quantities $x,y$ with a constant $c$, which is independent of the mesh parameters $h$, $\Delta t$, time instance $t_n$, and the position of $\Gamma$ over the background mesh. Similarly we give sense to $x\gtrsim y$; and $x\simeq y$ will mean that both $x\lesssim y$ and $x\gtrsim y$ hold.

In the analysis we require different domains stemming from the extension. We define strips which are \emph{sharp} in the sense that they do not include full elements. These are the boundary strips
\begin{equation}
\Sdh{n} :=\{ \bx \in \wOm \,: \dist(\bx,\Gamma_h^n) \leq \delta_h\} \quad\text{and}\quad \Sdhp{n} :=\{ \bx \in \wOm\setminus \Om{n}{h} \,: \dist(\bx,\Gamma_h^n) \leq \delta_h\}.
\end{equation}
Analogously we define $\Sd{n}$ and $\Sdp{n}$ to $\Om{n}{}$. Further, we define the overlaps $\Odhn{n} := \Sdh{n} \cup \Om{n}{h}$, $\Odn{n} := \Sd{n} \cup \Om{n}{}$.
We notice that $\Sdh{n}$ includes points from the interior of $\Om{n}{h}$ as well as points that are outside of $\Om{n}{h}$.
All elements that have some part in the strip $\Sdh{n}$ (or $\Sdhp{n}$) are collected in $\TS{n}$(or $\TSp{n}$).
We specify
\begin{equation} \label{e:delta}
  \delta_h = c_{\delta_h} \winfn ~\Delta t \quad \text{ with } \quad 1 < c_{\delta_h} < c_{\delta},
\end{equation}
and have that the size of the extension strip scales with $\delta_h$,  $| \Sdh{n} | \simeq \delta_h$.
The set of elements that have some part in $\Om{n}{h}$ is denoted by
\begin{equation*}
  \T^n:=\{ S\in\mathcal{T}_h\,:\, \meas_{d}(S\cap\Om{n}{h})>0\}.
\end{equation*}
We refer to Figures \ref{fig:discretedomains} and \ref{fig:discretedomains2} for sketches of the different domains and parts of the mesh.

In the analysis below we require that $\delta$ is sufficiently large so that for $n=1,\dots,N$
\begin{equation}\label{neigh_cond}
  \Odt{n} \subset\O{}(\Om{n}{})
  \qquad \text{and} \qquad \Om{n}{h} \subset \O{}(\Om{}{}(t)),~ t\in I_n.
\end{equation}
For the discrete extension layer, we have with \eqref{e:delta} that there holds
\begin{equation}\label{cond1}
  c_{\delta_h} \text{ sufficiently large implies }\Om{n}{h}
  \subset \Odt{n-1}, \quad n=1,\dots,N.
\end{equation}
This condition is the discrete analog to \eqref{ass1} and it is essential for the well-posedness of the method.

\begin{figure}
  \begin{center}
    \includegraphics[width=0.99\textwidth]{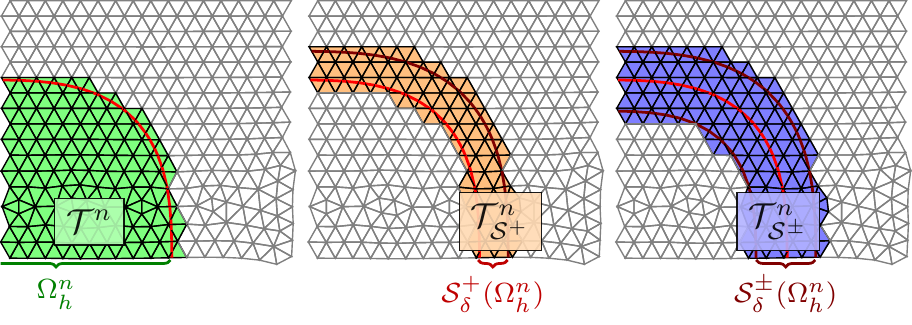}
  \end{center}
  \vspace*{-0.5cm}
  \caption{Sketch of discrete domains and different parts of the mesh.}
  \vspace*{-0.2cm}
  \label{fig:discretedomains2}
\end{figure}

\subsection{Geometry approximation} \label{sec:geomapprox}
We assume that the approximation of the geometry is of higher order in the sense that
\begin{equation}
  \dist(\Om{n}{},\Om{n}{h}) \lesssim h^{q+1},
\end{equation}
where $q$ is the geometry order of approximation and we assume that integrals on $\Om{n}{h},~n=1,\dots,N$ can be computed accurately.
Furthermore, we assume that there is a mapping $\Phi: \Odhn{n} \rightarrow \Odn{n}$ that allows to map from the approximated (extended) domain to the exact (extended) domain and assume  that the mapping $\Phi$ is well-defined, continuous and there holds $\Om{n}{} = \Phi(\Om{n}{h})$ and
$\Odn{n} = \Phi(\Odhn{n})$
and
    \begin{align} \label{eq:estPhi}
    \Vert \Phi - \operatorname{id} \Vert_{L^\infty(\Odhn{n})}
    \lesssim h^{q+1}, \quad
    \Vert D \Phi - I \Vert_{L^\infty(\Odhn{n})}
    \lesssim h^{q}.
    \end{align}
Further, for $h$ sufficiently small $\Phi$ is invertible.
Such a mapping has been constructed in \cite[Section 7.1]{gross2015trace} and \cite[Lemma 5.1]{olshanskii2016narrow} based on a level set based approximation of the geometries, cf. Remark~\ref{rem:lset} .
We use such a mapping to map from the discrete domain to the exact one. For $u \in V_h^n$ we define $\ul := u \circ \Phi^{-1}$.
Due to \eqref{eq:estPhi} we have that
\begin{subequations} \label{eq:equivvol}
\begin{align}
&  \Vert \ul \Vert_{\Odn{n}}^2 = \int_{\Odn{n}} (\ul)^2~dx = & \int_{\Odhn{n}} \underbrace{\operatorname{det}(D\Phi)}_{\simeq 1} u^2~dx &\simeq \Vert u \Vert_{\Odhn{n}}^2, &~~   \Vert u \Vert_{\Om{n}{h}}^2 &\simeq \Vert \ul \Vert_{\Om{n}{}}^2, \\
&  \text{ and similarly }
  & \Vert \nabla \ul \Vert_{\Odn{n}}^2 &\simeq \Vert \nabla u \Vert_{\Odhn{n}}^2, &~~   \Vert \nabla u \Vert_{\Om{n}{h}}^2 &\simeq \Vert \nabla \ul \Vert_{\Om{n}{}}^2.
\end{align}
\end{subequations}

\begin{remark}[Level set based domain descriptions] \rm \label{rem:lset}
  One popular method to obtain geometry approximations is based on level sets \cite{SethianBook}. Assume a level set function $\phi^n$ is known so that $\Om{n}{} = \{ \bx \in \wOm \mid \, \phi^n < 0 \}$. Further assume that $\phi^n$ is smooth and $\Vert \nabla \phi^n \Vert_2 \simeq 1$ close to the domain boundary $\partial \Om{n}{}$. In practice one typically only has an approximation $\phi_h^n$ to $\phi^n$ that may have been obtained from interpolation or solving a transport problem based on $\bw^e$ and an initial level set functions. Using this discrete approximation $\phi_h^n$ we then define $\Om{n}{h} = \{ \bx \in \wOm \mid \, \phi_h^n < 0\}$. If $\phi_h^n$ is a suitable good approximation to $\phi^n$ the approximation assumption \eqref{eq:estPhi} holds true where $q$ is the degree of the approximation for $\phi_h^n$. 
  Most often only the case $q=1$ is considered as only then $\Om{n}{h}$ is a polygonal domain which facilitates the implementation of numerical integration. Below, in the numerical examples we also restrict to $q=1$. However, we mention that also higher order geometrical accuracy can be realized for level set domains, cf. \cite{muller2013highly,lehrenfeld2015cmame,saye2015hoquad,fries2015,olshanskii2016numerical}.
\end{remark}

\subsection{Stability of discrete extensions through $s_h^n(\cdot,\cdot)$}
In this section we give some fundamental properties which hold for all variants of the stabilization $s_h^n(\cdot,\cdot)$ presented before. Let us mention that very recently in \cite{guerkanmassing2018} a similar analysis that unifies the properties of the ghost penalty versions from Section \ref{s:gplps} and Section \ref{s:gpdjump} has been use for stationary unfitted problems.

%
For all versions of the ghost penalty stabilizations mentioned above, we can split $s_h^n(\cdot,\cdot)$ into facet-contributions:
\begin{equation}
  s_h^n(\cdot,\cdot) = \frac{1}{h^2} \sum_{F\in \Fh^n} s_{h,F}^n(\cdot,\cdot),
\end{equation}
where $s_{h,F}^n(\cdot,\cdot)$ provides the following local stabilization property.
\begin{lemma} \label{lem:shf}
  Let $T_1 \in \TS{n}$ and $T_2 \in \Td{n}$, $T_1\neq T_2$ so that for $F = \overline{T_1}\cap\overline{T_2}$ there holds $\meas_{d-1}(F)>0$. Then we have for $~u|_{T_i} \in P_m(T_i),~i=1,2$ that there holds
\begin{subequations}
  \begin{align}
    \norm{u}_{T_1}^2 &\lesssim \norm{u}_{T_2}^2 + s_{h,F}^n(u,u),  \label{e:gplocal}\\
    \norm{\nabla u}_{T_1}^2 &\lesssim \norm{\nabla u}_{T_2}^2 + \frac{1}{h^2} s_{h,F}^n(u,u). \label{e:gplocalgrad}
  \end{align}
\end{subequations}
\end{lemma}
\begin{proof}
For the derivative jump version the results are given in \cite{burman2012fictitious} for $m=1$ and extended to the higher order case in \cite[Lemma 5.1]{massing2014stabilized}.
For $s_{h,F}^{n,\text{dir}}(\cdot,\cdot)$ we use \cite[Lemma 3.1]{preussmaster} for \eqref{e:gplocal} and turn our attention to \eqref{e:gplocalgrad}:
Note that $\nabla u|_{T_i} \in [P_{m-1}(T_i)]^d,~i=1,2,$ so that we can apply \eqref{e:gplocal} componentwise, i.e.
  \begin{equation}
    \norm{\nabla u}_{T_1}^2 \lesssim \norm{\nabla u}_{T_2}^2 + s_{h,F}^{n,\text{dir}}(\nabla u,\nabla u).
  \end{equation}
  Now applying an inverse inequality gives $s_{h,F}^{n,\text{dir}}(\nabla u,\nabla u) \lesssim \frac{1}{h^2} s_{h,F}^{n,\text{dir}}(u,u)$.
Finally, with
  \begin{align}
    s_{h,F}^{n,\text{dir}}(v,v) & = \sum_{i=1,2} \Vert v_1 - v_2 \Vert_{T_i}^2
    \lesssim
    \sum_{i=1,2} \Vert v_1 - \Pi_{\omega_F} v \Vert_{T_i}^2 + \Vert v_2 - \Pi_{\omega_F} v \Vert_{T_i}^2
    = \sum_{i=1,2} \Vert v_i - \Pi_{\omega_F} v \Vert_{\omega_F}^2 \nonumber \\
\label{e:lpsboundsdir}
    & \lesssim
    \sum_{i=1,2} \Vert v_i - \Pi_{\omega_F} v \Vert_{T_i}^2
    =
    \sum_{i=1,2} \Vert v - \Pi_{\omega_F} v \Vert_{T_i}^2
    =
      \Vert v - \Pi_{\omega_F} v \Vert_{\omega_F}^2
    = s_{h,F}^{n,\text{LPS}}(v,v),
  \end{align}
equations \eqref{e:gplocal} and \eqref{e:gplocalgrad} follow also for $s_{h,F}^{n,\text{LPS}}(\cdot,\cdot)$.
\end{proof}

We now want to apply this stabilizing mechanism globally.
To this end, we make an assumption on the meshes $\TSp{n}$ and $\Td{n}$ which we comment on in Remark \ref{rem:overlap}
\begin{assumption} \label{ass:overlap}
  To every element in $\TSp{n}$ we require an element in $\Td{n}\setminus\TSp{n}$ that can be reached by repeatedly passing through facets in $\Fh^n$. We assume that there is mapping that maps every element $T \in \TSp{n}$ to such a path with the following properties. The number of facets passed through during this path is bounded by $K \lesssim (1 + \frac{\delta_h}{h})$. Further, every uncut element $T \in \Td{n}\setminus\TSp{n}$ is the final element of such a path in at most $M$ of these paths where $M$ is a number that is bounded independently of $h$ and $\Delta t$.
\end{assumption}
\begin{remark}\label{rem:overlap}\rm
We briefly explain why Assumption \ref{ass:overlap} is reasonable if the smooth domain boundary $\Gamma^{n}$ is sufficiently well-resolved by the mesh. To this end we construct a mapping between elements: $B: \TSp{n} \to \Td{n}\setminus\TSp{n}$.
We take an inner point $\bx_T$ (e.g. the circumcenter) of an element in $T \in \TSp{n}$ and map it by a distance of $\delta_h$ towards the interior, $\by_{T'} := \bx_T + \delta_h (\bp(\bx_T) - \bx_T)$ where $\bp$ is the closest point projection on $\Gamma^n$.
There is an element $T' \subset \Td{n}\setminus\TSp{n}$ that contains $\by_{T'}$ or can be reached from $\by_{T'}$ by passing only through a few ($\lesssim 1$) facets in $\Fh^n$. Hence, due to shape regularity, the number $K$ of facets in $\Fh^n$ that are intersected by the path $\{\bx_T + s (\bp(\bx_T)-\bx_T),~s\in[0,\delta_h]\}$ are bounded by $c (h+\delta_h)/h$.
Due to the geometrical construction of $B$ and the assumed resolution of the boundary we further have that only a few elements in $\TSp{n}$ will be mapped to the same element $T' \in \Td{n} \setminus \TSp{n}$, i.e. $|B^{-1}(T')| \leq M \lesssim 1,~\forall T' \in \Td{n} \setminus \TSp{n}$.
\end{remark}

With this assumption we have control on the overlap to obtain the following result.
\begin{lemma} \label{lem:gp}
  Under Assumption \ref{ass:overlap}, there holds for $u \in V_h^n$:
  \begin{subequations}
  \begin{align}
    \norm{u}_{\Odhn{n}}^2 &\leq    \norm{u}_{\Odt{n}}^2 \lesssim \norm{u}_{\Om{n}{h}}^2 + K\, h^2 \ s_{h}^n(u,u), \label{e:gpglobal}\\
    \norm{\nabla u}_{\Odhn{n}}^2 &\leq    \norm{\nabla u}_{\Odt{n}}^2 \lesssim \norm{\nabla u}_{\Om{n}{h}}^2 + K \  s_{h}^n(u,u). \label{e:gpglobalgrad}
  \end{align}
\end{subequations}
\end{lemma}
\begin{proof}
We start with \eqref{e:gpglobal}. First, we notice that there holds
\begin{equation}
  \norm{u}_{\Odhn{n}}^2 \leq \norm{u}_{\Odt{n}}^2 = \sum_{T\in\Td{n}} \norm{u}_T^2 = \sum_{T\in\TSp{n}} \norm{u}_T^2 + \sum_{T\in\Td{n} \setminus \TSp{n}} \norm{u}_T^2 \leq \sum_{T\in\TSp{n}} \norm{u}_T^2 + \norm{u}_{\Om{n}{h}}^2.
\end{equation}
Now, we repeatedly apply the previous lemma to pass from each $T \in \TSp{n}$ to a $T' \in \Td{n} \setminus \TSp{n}$. Due to Assumption \ref{ass:overlap} every element in $T \in \Td{n} \setminus \TSp{n}$ will appear only $M \lesssim 1$ times and every facet $F \in \Fh^n$ will only appear $K$ times. Analogously, \eqref{e:gpglobalgrad} follows.
\end{proof}

In the analysis of the time stepping, the critical region to control is the extension strip $\Sdhp{n}$. We can bound the $L^2$ norm on this strip by norms on $\Odhn{n}$ and a scaling with $\delta_h$.
\begin{lemma} \label{lem:strip1} For $u \in H^1(\Odhn{n})$ and any $\ep > 0$ there holds
 \begin{equation} \label{e:strip}
\|u\|_{\Sdhp{n}}^2  \lesssim \delta_h (1 + \ep^{-1}) \|u\|_{\Odhn{n}}^2 + \delta_h \ep \| \nabla u\|_{\Odhn{n}}^2.
\end{equation}
\end{lemma}
\begin{proof}
  We notice that $\Phi$ in Section~\ref{sec:geomapprox} maps $\Sdhp{n}$ on  $\Sdp{n}$.
  By applying the transformation rules as in \eqref{eq:equivvol} it suffices to show
 \begin{equation} \label{e:striptransformed}
\|\ul\|_{\Sdp{n}}^2  \lesssim \delta_h (1 + \ep^{-1}) \|\ul\|_{\Odn{n}}^2 + \delta_h \ep \| \nabla \ul\|_{\Odn{n}}^2
\end{equation}
for $\ul = u \circ \Phi^{-1}$, $u \in H^1(\Odhn{n})$. This however has been shown in Lemma \ref{l_est1} (with only a different size of the extension strip).
\end{proof}
To bound the norms on $\Odhn{n}$ by corresponding norms in $\Om{n}{h}$ we finally make use of the stabilization and obtain as a direct consequence of Lemma \ref{lem:gp} and \ref{lem:strip1}:
\begin{lemma} \label{lem:strip2} Using Assumption \ref{ass:overlap}, for $u \in V_h^n$ and any $\ep > 0$ there holds
 \begin{align} \label{e:striph}
   \|u\|_{\Sdhp{n}}^2 \lesssim  \quad
   & \delta_h \ (1 + \ep^{-1})  \|u\|_{\Om{n}{h}}^2 + \delta_h \ \ep \| \nabla u\|_{\Om{n}{h}}^2
   +
     \delta_h K \ ( (1+\ep^{-1})h^2 + \ep) s_h^n(u,u). 
 \end{align}
 As a direct consequence we have for a constant $c_{L\ref{lem:strip2}}$ independent of $h$ and $\Delta t$
 \begin{align} \label{e:extendedh}
   \|u\|_{\Odhn{n}}^2 \leq
   ( 1 + & c_{L\ref{lem:strip2}a}(\ep) \, \Delta t) \|u\|_{\Om{n}{h}}^2 + c_{L\ref{lem:strip2}b}(\ep) \, \alpha \Delta t  \| \nabla u\|_{\Om{n}{h}}^2 + c_{L\ref{lem:strip2}c}(\ep,h)\,  \Delta t \, K s_h^n(u,u), \end{align}
 where
 $c_{L\ref{lem:strip2}a}(\ep) = c_{L\ref{lem:strip2}} c_{\delta_h} \winfn (1 + \ep^{-1})$,  $c_{L\ref{lem:strip2}b}(\ep) = c_{L\ref{lem:strip2}} c_{\delta_h} \winfn \ep/\alpha$ and
$c_{L\ref{lem:strip2}c}(\ep,h) = c_{L\ref{lem:strip2}} c_{\delta_h} \winfn (\ep +h^2 + h^2 \ep^{-1})$.
\end{lemma}
Finally, we treat consistency aspects of the stabilization:
\begin{lemma}\label{lem:est_s}
  For $s_h^n \in \{ s_h^{n,\text{dir}}, s_h^{n,\text{LPS}}, s_h^{n,\text{djmp}} \}$ as in the Sections
  \ref{s:gpdirect}, \ref{s:gplps} and \ref{s:gpdjump}, respectively, and $w \in H^{m+1}(\Odt{n}),~n=1,\dots, N$, there holds
  \begin{subequations}
  \begin{equation} \label{eq:shnw}
    s_h^n(w,w) \lesssim h^{2m} \Vert w \Vert_{H^{m+1}(\Odt{n})}^2.
  \end{equation}
  Let $\mathcal{I}$ be the Lagrange interpolation operator. Then for $w \in H^{m+1}(\Odt{n}),~n=1,\dots, N$,
  \begin{equation} \label{eq:shnIw}
    s_h^n(w - \mathcal{I} w, w - \mathcal{I} w) \lesssim h^{2m} \Vert w \Vert_{H^{m+1}(\Odt{n})}^2.
  \end{equation}
\end{subequations}
\end{lemma}
\begin{proof}
  We start with \eqref{eq:shnw}. There holds $s_h^{n,\text{djmp}}(w,w) = 0$ due to the continuity of the corresponding (higher order) derivatives, i.e. we only have to consider $s_h^n \in \{ s_h^{n,\text{dir}}, s_h^{n,\text{LPS}}\}$. We start with $s_h^{n,\text{dir}}$ and consider one facet contribution for $F \in \Fh^n$. Let $w_{i,h} = \mathcal{E}^P \Pi_{T_i} w|_{T_i},~i=1,2$, then
  \begin{align}
    s_{h,F}^{n,\text{dir}}(w,w)
    & = \norm{w_{1,h} - w_{2,h}}_{\omega_F}^2 \lesssim \sum_{i=1,2} \norm{w_{i,h} - \Pi_{\omega_F} w}_{\omega_F}^2 = \sum_{i=1,2} \sum_{j=1,2} \norm{w_{i,h} - \Pi_{\omega_F} w}_{T_j}^2 \nonumber \\
    & \stackrel{(\ast)}{\lesssim} \sum_{i=1,2} \norm{w_{i,h} - \Pi_{\omega_F} w}_{T_i}^2
      \lesssim \norm{w - \Pi_{\omega_F} w}_{\omega_F}^2 + \sum_{i=1,2} \norm{w - w_{i,h}}_{T_i}^2 \\
    & \lesssim s_{h,F}^{n,\text{LPS}}(w,w) + h^{2m+2} \sum_{i=1,2} \norm{w}_{H^{m+1}(T_i)}^2, \nonumber
  \end{align}
  where we used shape regularity in $(\ast)$ to bound the $L^2$ norm of a polynomial on $T_1$ by its $L^2$ norm on $T_2$ (and vice versa).
  Finally, a standard approximation result of the $L^2$ projection gives
  $\norm{w - \Pi_{\omega_F} w}_{\omega_F}^2 \lesssim h^{2k+2} \norm{w}_{H^{k+1}(\omega_F)}^2$. Adding over all facets and noting that we have a finite overlap of at most $d+1$ contributions per element concludes the proof.

  We turn our attention to \eqref{eq:shnIw} and start with $s_h^{n,\text{djmp}}(\cdot,\cdot)$.
Let $e_w = w - \mathcal{I} w$ and $T_F$ be an element so that $F \subset \partial T_F$. With trace inequalities we obtain
\begin{equation}
  s_h^{n,\text{djmp}}(e_w,e_w) \le \sum_{F \in \Fh^n} \sum_{k=0}^m \frac{h^{2k-1}}{k!^2}
 \left(  h^{-1} \Vert D^k e_w \Vert_{T_F}^2 + h \Vert D^{k+1} e_w \Vert_{T_F}^2 \right).
\end{equation}
The claim follows for $s_h^{n,\text{djmp}}$ from $ \Vert D^k e_w \Vert_{T_F}\lesssim h^{m+1-k} \Vert w \Vert_{H^{m+1}(T_F)},~k=0,\dots,m+1$. Now consider $s_h^n \in \{ s_h^{n,\text{dir}}, s_h^{n,\text{LPS}} \}$. With the stability of the $L^2$ projections and the polynomial extension operator $\mathcal{E}^P: P_m(T_1) \to P_m(T_2)$, one easily checks
\begin{equation}
  s_h^n(e_w,e_w) \lesssim h^{-2} \Vert e_w \Vert_{\Odt{n}}^2 \lesssim h^{-2} h^{2m+2} \Vert w \Vert_{H^{k+1}(\Odt{n})}^2.
\end{equation}
\end{proof}

\begin{remark}[Algebraic stability] \label{rem:linalg} \rm
  The stabilization bilinear form $s_h^n(\cdot,\cdot)$ is based on the active mesh rather then the concrete boundary/mesh intersection. This results in stability properties which are robust with respect to the cut positions. Furthermore, this robustness also carries over to the conditioning of linear systems. In the original paper \cite{B10} it was already shown for an elliptic model problem that the condition number can be bounded independent of the cut position. We notice that this result can also be carried over to the linear systems arising from \eqref{e:unfFEM1}.
\end{remark}

\subsection{Stability analysis} \label{s_stab}

\begin{theorem}\label{Th1}
  Under Assumption \ref{ass:overlap}, sufficiently large $c_\gamma$ in \eqref{e:cgamma}, and $\Delta t$ sufficiently small, the solution of \eqref{e:unfFEM1} satisfies the following estimate:
  \begin{equation}\label{FE_stab1}
    \|u_h^k\|^2_{\Om{k}{h}} + {\Delta t} \sum_{n=1}^{k}\left(  \alpha/2 \, \norm{\nabla u_h^n}_{\Om{n}{}}^2 + \gamma_s s_h^n(u_h^n,u_h^n) \right)
    \leq \exp(c_{T\ref{Th1}} t_k) \enorm{u_h^0}_0 .
\end{equation}
with $c_{T\ref{Th1}}$ independent of $h$, $\Delta t$ and $k=1,\dots,N$ and $\enorm{\cdot}_0$ as in \eqref{e:enorm}.
\end{theorem}
\begin{proof}
  We test \eqref{e:unfFEM1} with $u_h^n$ and multiply by $2 \Delta t$ which yields:
  \begin{equation}
    \norm{u_h^n}_{\Om{n}{h}}^2 + \norm{u_h^n - u_h^{n-1}}_{\Om{n}{h}}^2 + 2\Delta t \, \ahn(u_h^n,u_h^n) + 2\Delta t \, \gamma_s s_h^n(u_h^n,u_h^n) = \norm{u_h^{n-1}}_{\Om{n}{h}}^2.
  \end{equation}
Using $\norm{u_h^n - u_h^{n-1}}_{\Om{n}{h}}^2 > 0$, the lower bound on $\ahn(\cdot,\cdot)$, cf. \eqref{e:anhbound}, and Lemma \ref{lem:strip2}, we get
  \begin{align}
    (1 - & 2 \xi_h \Delta t ) \norm{u_h^n}_{\Om{n}{h}}^2 + \Delta t \alpha \norm{\nabla u_h^n}_{\Om{n}{h}}^2 + 2 \gamma_s \Delta t s_h^n(u_h^n,u_h^n)
    \leq
             \norm{u_h^{n-1}}_{\Om{n}{h}}^2 \leq \norm{u_h^{n-1}}_{\Odhn{n-1}}^2      \\
    \leq
         &
           ( 1 + c_{L\ref{lem:strip2}a}(\ep) \, \Delta t ) \|u_h^{n-1}\|_{\Om{n-1}{h}}^2
           + c_{L\ref{lem:strip2}b}(\ep) \, \Delta t \alpha \| \nabla u_h^{n-1}\|_{\Om{n-1}{h}}^2
    +  c_{L\ref{lem:strip2}c}(\ep,h)  \Delta t K \ s_h^{n-1}(u_h^{n-1},u_h^{n-1}). \nonumber
  \end{align}
  We choose $\varepsilon \leq \alpha / (2 c_{L\ref{lem:strip2}} c_{\delta_h} \winfn)$ so that $c_{L\ref{lem:strip2}b}(\ep) \leq 1/2$ and $c_{L\ref{lem:strip2}a}(\ep)$ and  $c_{L\ref{lem:strip2}c}=c_{L\ref{lem:strip2}c}(\ep,h)$ are bounded independent of $h$ and $\Delta t$.
  Further, we assume $\gamma_s \geq c_{L\ref{lem:strip2}c} K $.
  Summing up  over $n=1,\dots,k,~k\leq N$ yields
  \begin{align}
    (1-& 2\xi_h \Delta t ) \norm{u_h^k}_{\Om{k}{h}}^2 + \alpha/2 \, \Delta t  \sum_{n=1}^k  \norm{\nabla u_h^n}_{\Om{n}{}}^2 + \gamma_s \Delta t \sum_{n=1}^k s_h^n(u_h^n,u_h^n)
    \nonumber \\
    & \leq \norm{u^0}_{\Om{0}{}}^2 + (c_{L\ref{lem:strip2}a} + 2 \xi_h) \Delta t \sum_{n=0}^{k-1} \| u^n \|_{\Om{n}{}}^2 +  \gamma_s \Delta t s_h^0(u_h^0,u_h^0) + \alpha/2 \, \Delta t  \norm{\nabla u_h^0}_{\Om{0}{}}^2.
  \end{align}
  Now we can apply Gronwall's Lemma with $\xi_h \Delta t \leq \frac14$ and obtain the result with
  $c_{T\ref{Th1}} = c_{L\ref{lem:strip2}a} + \xi_h$.
\end{proof}

\subsection{Consistency estimates} \label{sec:consist}

Testing \eqref{transport} with $v_h^l = v_h \circ \Phi^{-1},~v_h \in V_h^n$, where $\Phi$ as in
Section \ref{sec:geomapprox}, we see that any smooth solution to \eqref{transport} satisfies
  \begin{equation} \label{e:exid}
    \int_{\Om{n}{}} \partial_t u (t_n) v_h^l\, dx + a^n(u(t_n),v_h^l) = 0 \quad \text{ for all } v_h^l = v_h \circ \Phi^{-1},~v_h \in V_h^n.
  \end{equation}
%
For the solution $u$ to \eqref{transport}  we identify its extension $\E{}u$, cf. \eqref{e:extensiont_cont}, from $\Qs$ to $\O(\Qs)$ with $u$.

Thanks to \eqref{neigh_cond} $u^{n-1} = u(t_{n-1})$ is  well-defined on $\Om{n}{h}$ and  $u^{n} = u(t_{n})$ is  well-defined on $\Odt{n}$.
Let $\err^n = u^n - u_h^n$, subtracting \eqref{e:unfFEM1} from \eqref{e:exid} we obtain the error equation
\begin{equation} \label{e:erreq}
  \int_{\Om{n}{h}} \frac{\err^n - \err^{n-1}}{\Delta t} v_h dx + \ahn(\err^n,v_h) + \gamma_s s_h^n(\err^n,v_h) = \consist(v_h),
\end{equation}
with (again $v_h^l = v_h \circ \Phi^{-1}$) \vspace*{-0.7cm}
\begin{align*}
  \consist(v_h) :=
  & \hphantom{+} \overbrace{
    \int_{\Om{n}{}} u_t(t_n) v_h^l dx - \int_{\Om{n}{h}} \frac{u^n-u^{n-1}}{\Delta t} v_h dx
    }^{I_1}
  + \overbrace{ \vphantom{\int_{\Om{n}{}}} a^n(u^n,v_h^l) - \ahn(u^n,v_h)
    }^{I_2}
  + \overbrace{ \vphantom{\int_{\Om{n}{}}} \gamma_s s_h^n(u^n,v_h)
    }^{I_3}.
\end{align*}

\begin{lemma}\label{l_consist} Assume $u\in W^{2,\infty}(\Qs)\cap L^\infty(0,T;H^{m+1}(\Omega(t)))$,
then the consistency error has the bound
\begin{equation}\label{est:consist}
  |\consist(v_h)|\lesssim (\Delta t+h^q+h^m K^{\frac12}) \, (\norm{u}_{W^{2,\infty}(\Qs)}+\sup_{t\in[0,T]}\norm{u}_{H^{m+1}(\Omega(t))}) \,
  \enorm{v_h}_n.
\end{equation}
\end{lemma}
\begin{proof} We treat $\consist(v_h)$ term by term, starting with $I_1$:
\[
I_1=  \int_{\Om{n}{h}}\int_{t_{n-1}}^{t_n}\frac{t-t_{n-1}}{\Delta t} u_{tt} \,dt\,v_h\, dx-\int_{\Om{n}{h}}u_t(t_n) v_h\, dx+\int_{\Om{n}{}}u_t(t_n) v_h^l\, dx.
\]
We have with $\Om{n}{h} \in \O(\Om{}{}(t)),~t\in I_n$ and \eqref{u_bound_a3}
\begin{equation*}
\left| \int_{\Om{n}{h}}\int_{t_{n-1}}^{t_n}u_{tt} \frac{t-t_{n-1}}{\Delta t}\,dt\,v_h\, ds\right| \le \frac12 \Delta t\|u_{tt}\|_{L^\infty(\O(\Qsn))}\|v_h\|_{L^1(\Om{n}{h})} \lesssim \Delta t\|u\|_{W^{2,\infty}(\Qs)}\|v_h\|_{\Om{n}{h}},
\end{equation*}
and
\begin{multline*}
  \left|\int_{\Om{n}{h}}u_t(t_n) v_h\, dx-\int_{\Om{n}{}}u_t(t_n) v_h^l\, dx\right|
  = \left|\int_{\Om{n}{h}}u_t(t_n)-(u_t \circ \Phi)(t_n) (1-\det(D\Phi)) v_h\, dx\right|\\
   \lesssim h^{q}(\|\nabla u_t\|_{L^\infty(\O(\Om{n}{}))}\|v_h\|_{\Om{n}{h}} \lesssim h^{q}\| u\|_{W^{2,\infty}(\Qs)}\|v_h\|_{\Om{n}{h}},
\end{multline*}
where we used the change of variables, the second bound in \eqref{eq:estPhi} and
\[|u_t(x,t_n)-(u_t \circ \Phi)(x,t_n)|\le \|\nabla u_t\|_{L^\infty(\O(\Om{n}{}))} |x-\Phi(x)|\]
with the first bound in \eqref{eq:estPhi}.
The bound for $I_2$ analogously follows from the differentiation chain rule and \eqref{eq:estPhi}, see, e.g., \cite[Lemma 7.4]{gross2015trace},
\begin{equation*}
  |I_2|  \lesssim h^{q}\|u\|_{W^{2,\infty}(\Qs)}\|v_h\|_{H^1(\Om{n}{h})}.
\end{equation*}
For the third term, $I_3$, we first use the Cauchy--Schwarz inequality and further the result of Lemma~\ref{lem:est_s},
\begin{equation*}
s_h^n(u^n,v_h)\le s_h^n(u^n,u^n)^\frac12 s_h^n(v_h,v_h)^\frac12 \lesssim h^{m}\|u\|_{H^{m+1}(\Odt{n})}s_h^n(v_h,v_h)^\frac12
\lesssim h^{m}\|u\|_{H^{m+1}(\Om{n}{})}s_h^n(v_h,v_h)^\frac12.
\end{equation*}
In the last bound we used $\Odt{n}\subset\O(\Om{n}{})$ and \eqref{u_bound_a1}.
\end{proof}

We notice that the latter part in the consistency error, $h^m K^{\frac12}$, vanishes for the derivative jump formulation as $s_h^n(u,v_h) = 0$ for all $u \in H^{m+1}(\Odt{n})$, $v_h \in V_h^n$.

\subsection{Error estimate in the energy norm} \label{sec:aprioriest}
 We let $u = \mathcal{I} u^n \in V_h^n$ be the Lagrange interpolant  for $u^n$ in $\Odt{n}\subset\O(\Om{n}{})$; we assume $u^n$  sufficiently smooth so that the interpolation is well-defined. Following standard lines of argument, we split $\err^n$ into finite element and approximation parts,
\[
\err^n=\underset{\mbox{$e^n$}}{\underbrace{(u^n-u^n_I)}}+\underset{\mbox{$e^n_h \in V_h^n$}}{\underbrace{(u^n_I-u^n_h)}}.
\]
Equation \eqref{e:erreq} yields
\begin{equation}\label{e:err1}
\int_{\Om{n}{h}}\left(\frac{e^n_h-e^{n-1}_h}{\Delta t} \right) v_h\,ds+ a^h_n(e_h^n,v_h)+\gamma_s s_h^n(u_h^n,v_h) = \interpol(v_h)+\consist(v_h),\quad \forall~v_h\in V^n_h,
\end{equation}
with the interpolation term
\[
\interpol(v_h)=-\int_{\Om{n}{h}}\left(\frac{e^n-e^{n-1}}{\Delta t} \right) v_h\,ds_h- \ahn(e^n,v_h)-\gamma_s s_h^n(e^n,v_h).
\]
We give the estimate for the interpolation terms in the following lemma.

\begin{lemma}\label{l_interp} Assume $u\in L^\infty(0,T;H^{m+1}(\Om{}{}(t)))$ and $u_t\in L^\infty(0,T;H^{m}(\Om{}{}(t)))$,
then it holds
\begin{equation}\label{est_inter}
  |\interpol(v_h)|\lesssim h^mK^{\frac12} \,\sup_{t\in[0,T]}(\|u\|_{H^{m+1}(\Om{}{}(t))}+\|u_t\|_{H^{m}(\Om{}{}(t))}) \, \enorm{v_h}_n.
\end{equation}
\end{lemma}
\begin{proof}
We use standard interpolation properties of polynomials to conclude
\begin{equation}\label{eq:interp}
\|e^n\|_{\Odt{n}}+h\|\nabla e^n\|_{\Odt{n}}\lesssim h^{m+1}\|u^n\|_{H^{m+1}(\O(\Om{n}{}))}\lesssim h^{m+1}\|u^n\|_{H^{m+1}(\Om{n}{})}.
\end{equation}
The last inequality is thanks to \eqref{u_bound_a1}.
On $\Om{n}{h}$ we extend $u_I^h$ for all $t\in[t_{n-1},t_n]$  as the Lagrange interpolant of $u(t)$ in all nodes from $\Odt{n}$ so that $u^n_I=u_I^{n-1}$ on $\Om{n}{h}$ for $t=t_{n-1}$. Since $(u^n_I)_t$ appears to be the nodal interpolant for $u_t$, we have with \eqref{u_bound_a1} and $\Odt{n} \subset \O({\Om{}{}(t)}),~ t\in I_n$, that
\begin{equation}\label{eq:interp2}
\|e^n_t\|_{\Om{n}{h}}\lesssim h^{m}\|u_t\|_{H^{m}(\Odt{n})} \lesssim h^{m}\|u_t\|_{H^{m}(\O({\Om{}{}(t)})} \lesssim h^m \|u_t\|_{H^{m}(\Om{}{}(t))}, \quad t \in I_n.
\end{equation}
We treat the first term in $\interpol(v_h)$ using Cauchy--Schwarz, \eqref{eq:interp2}, and \eqref{u_bound_b},
\begin{equation}\label{e:err2}
\begin{split}
\left|\int_{\Om{n}{h}}\right.&\left.\left(\frac{e^n-e^{n-1}}{\Delta t} \right) v_h\,ds_h\right|\le \left\|\frac{e^n-e^{n-1}}{\Delta t}\right\|_{\Om{n}{h}}\|v_h\|_{\Om{n}{h}} 
 =|\Delta t|^{-1} \left\|\int^{t_n}_{t_{n-1}} e_t(t')\,dt'\right\|_{\Om{n}{h}}\|v_h\|_{\Om{n}{h}}\\&
 \le |\Delta t|^{-\frac12}\left(\int^{t_n}_{t_{n-1}}\| e_t(t')\|_{\Om{n}{h}}^2\,dt'\right)^{\frac12}\|v_h\|_{\Om{n}{h}}
 \lesssim \,h^m\sup_{t\in[t_{n-1},t_n]} \| u_t\|_{H^{m}(\Odt{n})}\|v_h\|_{\Om{n}{h}}\\
    & \lesssim\,h^{m}\sup_{t\in[0,T]}(\|u\|_{H^{m+1}(\Om{}{}(t))}+\|u_t\|_{H^{m}(\Om{}{}(t))})\|v_h\|_{\Om{n}{h}}.
\end{split}
\end{equation}
We handle the term $\ahn(e^n,v_h)$ in a straight-forward way using the Cauchy-Schwarz inequality and \eqref{eq:interp}:
\begin{equation*}
  |\ahn(e^n,v_h)|
  \lesssim h^m\|u^n\|_{H^{m+1}(\Om{n}{})}(\|v_h\|_{\Om{n}{h}}+\alpha^{\frac12}\|\nabla v_h\|_{\Om{n}{h}}).
\end{equation*}
The stabilization term is treated using the Cauchy--Schwarz inequality and Lemma~\ref{lem:est_s},
\begin{equation*}
\begin{split}
  s_h^n(e^n,v_h)&\le s_h^n(e^n,e^n)^\frac12 s_h^n(v_h,v_h)^\frac12
  \lesssim h^{m}\|u\|_{H^{m+1}(\Om{n}{})}s_h^n(v_h,v_h)^\frac12.
\end{split}
\end{equation*}
We summarize the above bounds into the estimate of the interpolation term as in \eqref{est_inter}.
\end{proof}

\begin{theorem}\label{Th2}
We make Assumption \ref{ass:overlap}, and assume $c_\gamma$ in \eqref{e:cgamma} to be sufficiently large, $\Delta t$ sufficiently small,  $u$ the  solution to \eqref{transport}, $u\in W^{2,\infty}(\Qs)\cap L^\infty(0,T;H^{m+1}(\Om{}{}(t)))$ and $u_t\in L^\infty(H^{m}(0,T;\Om{}{}(t)))$, $\Psi$ to be sufficiently smooth.
  For  $u_h^n$, $n=1,\dots,N$, the finite element solution of \eqref{e:unfFEM1},
  and $\err^n = u_h^n - u^n$  the following error estimate holds:
  \begin{equation}\label{FE_est1}
    \|\err^n\|^2_{\Om{n}{h}}+\frac{\Delta t}{2} \sum_{k=1}^{n}\! \left( \frac{\alpha}{2}\|\nabla\err^k \|^2_{\Om{k}{h}}
      +\gamma_s s_h^n(\err^n,\err^n)\right)
    \lesssim \exp(c_{T\ref{Th2}}t_n) R(u) (\Delta t^2+h^{2q}+h^{2m}K ),
  \end{equation}
  with $R(u) := \sup_{t\in[0,T]}(\|u\|_{H^{m+1}(\Om{}{}(t))}^2+\|u_t\|_{H^{m}(\Om{}{}(t))}^2)+\|u\|_{W^{2,\infty}(\Qs)}^2$ and $c_{T\ref{Th2}}$ independent of $h$, $\Delta t$, $n$ and of the positions of $\Omega_h$ over the background mesh.
\end{theorem}
\begin{proof}
We set $v_h=2\Delta t e^n_h$ in \eqref{e:err1}. This gives
\begin{equation*}
\|e_h^n\|^2_{\Om{n}{h}}- \|e_h^{n-1}\|^2_{\Om{n}{h}} +\|e_h^n-e_h^{n-1}\|^2_{\Om{n}{h}}+{2\Delta t} \ahn(e_h^n,e_h^n)+{2\Delta t}\gamma_s s^n_h(e_h^n,e_h^n)=2\Delta t(\interpol(e_h)+\consist(e_h)).
\end{equation*}
Repeating the arguments as in the proof of Theorem~\ref{Th1}, we get
\begin{align}\label{aux1}
    (1-& 2 \xi_h \Delta t ) \norm{e_h^k}_{\Om{k}{h}}^2 + \frac{\alpha}{2} \, \Delta t  \sum_{n=1}^k  \norm{\nabla e_h^n}_{\Om{n}{h}}^2 + \gamma_s \Delta t \sum_{n=1}^k s_h^n(u_h^n,u_h^n) \\
    & \leq \norm{e^0}_{\Om{0}{h}}^2 + (c_{L\ref{lem:strip2}a} + 2 \xi_h) \Delta t \sum_{n=0}^{k-1} \| e^n \|_{\Om{n}{h}}^2 +  \gamma_s \Delta t s_h^0(e_h^0,e_h^0) + \frac{\alpha}{2} \, \Delta t  \norm{\nabla e_h^0}_{\Om{0}{h}}^2
    + 2\Delta t  \sum_{n=1}^k(\interpol(e_h^n)+\consist(e_h^n)). \nonumber
\end{align}
To estimate the interpolation and consistency terms, we apply Young's inequality to the right-hand sides of \eqref{est:consist} and \eqref{est_inter} yielding
\begin{align*}
  2 \Delta t (\consist(e_h^n) + \interpol(e_h^n))&
  \le c\,\Delta t (\Delta t^2+h^{2q}+h^{2m}K) R(u)
                           +\frac{\Delta t}{2}
\left(\|e_h^n\|_{\Om{n}{h}}^2+\frac{\alpha}{2}\|\nabla e_h^n\|_{\Om{n}{h}}^2 +\gamma_s s_h^n(e_h^n,e_h^n) \right),
\end{align*}
with a constant $c$ independent of $h$, $\Delta t$, $n$ and of the position of the surface over the background mesh. Substituting this in \eqref{aux1} and  noting $e^0_h=0$ in $\Odt{0}$ we get
\begin{align}
    (1- 2 \xi_h \Delta t ) \norm{e_h^k}_{\Om{k}{h}}^2 & + \frac12 \left( \frac{\alpha}{2} \, \Delta t  \sum_{n=1}^k  \norm{\nabla e_h^n}_{\Om{n}{}}^2 + \gamma_s \Delta t \sum_{n=1}^k s_h^n(u_h^n,u_h^n) \right)
    \nonumber \\ \nonumber
    & \leq (c_{L\ref{lem:strip2}a} + 2 \xi_h + \frac12) \Delta t \sum_{n=0}^{k-1} \| e^n \|_{\Om{n}{}}^2 +cR(u) (\Delta t^2+h^{2q}+h^{2m}K).
  \end{align}
We apply the discrete Gronwall inequality with $\xi_h \Delta t \leq 1/4$ to get
\begin{align}
    \|e_h^k & \|^2_{\Om{k}{h}}+\frac12\sum_{n=1}^{k}{\Delta t} \left( \frac{\alpha}{2}\|\nabla e_h^n\|^2_{\Om{n}{h}}
              +\gamma_s s_h^n(e_h^n,e_h^n)\right)
              \lesssim \exp(c_{T\ref{Th2}} t_k)
                R(u)(\Delta t^2+h^{2q}+h^{2m}K ) =: \exp(c_{T\ref{Th2}} t_k) Q_{e}.
                \nonumber
\end{align}
Now the triangle inequality,  \eqref{eq:interp} and \eqref{aux1} give
\begin{equation*}
\begin{split}
    \|\err^k\|^2_{\Om{k}{h}}&+\frac12\sum_{n=1}^{k}{\Delta t} \left( \frac{\alpha}{2}\|\nabla\err^n\|^2_{\Om{n}{h}}
      +\gamma_s s_h^n(\err^n,\err^n)\right)
    \\ &
    \leq
    \exp(c_{T\ref{Th2}} t_k)Q_{e} +\|e^k\|^2_{\Om{k}{h}}+\frac12\sum_{n=1}^{k}{\Delta t} \left( \alpha\|\nabla e^n\|^2_{\Om{n}{h}}
      +\gamma_s s_h^n(e^n,e^n)\right)
    \\ &
    \lesssim \exp(c_{T\ref{Th2}} t_k)Q_{e}+ \sup_{n=1,\dots,k}\norm{u}_{H^{m+1}(\Om{n}{})} h^{2m}K.
\end{split}
\end{equation*}
  This completes the proof.
\end{proof}

\begin{remark}[Extension to BDF2] \label{rem:bdf2} \rm
To keep the analysis manageable, we restricted to the backward Euler discretization. However, the method is easily extendable to higher order time stepping methods.  For example, it is straightforward to extend the method to the second order accurate in time BDF2 scheme.  Indeed, the finite difference stencil for the time derivative is changed from $\frac{u^n-u^{n-1}}{\Delta t}$ to $\frac{3u^n-4u^{n-1}+u^{n-2}}{2\Delta t} $ in the semi-discrete method in
\eqref{e:ImEuler} and for the fully discrete method in \eqref{e:unfFEM1}.
Accordingly, the width of the neighborhood extension has to be increased so that $\Om{n}{} \subset \mathcal{O}(\Om{n-1}{}) \cap \mathcal{O}(\Om{n-2}{})$ and $\Om{n}{h} \subset \mathcal{O}(\Om{n-1}{h}) \cap \mathcal{O}(\Om{n-1}{h})$. This is done by changing  $\delta_n$ in \eqref{e:delta} to $\delta_n = 2 c_\delta \winfn \Delta t$. Further, in the proof of the coercivity in the (spatially) continuous and discrete setting we have to change the time step restrictions \eqref{eq:xi} and \eqref{eq:xih} according to the changed coefficient in the BDF formula. The Gronwall-type arguments in Section \ref{s:stab:semi-disc} and in Theorem \ref{Th1} have to be replaced with corresponding versions for the BDF scheme. To handle the time derivative terms, one can use the polarization identity (6.33) from \cite{Ern04}. Finally, the consistency analysis in Section~\ref{sec:consist} can then be improved, specifically the term $I_1$ leading to a higher order (in $\Delta t$) estimate in Lemma \ref{l_consist} and Theorem \ref{Th2}.
\end{remark}


\section{Numerical experiments}\label{s:Numerics}
In this section we present numerical experiments for the method proposed and analyzed before. First, we introduce the general setup of the experiments and define the parameters that are shared or varied between the experiments (Section \ref{s:numex:setup}) before we investigate the performance of the method for moving domain problems, cf. Section
\ref{s:numex:ex1}, \ref{s:numex:ex2} and \ref{s:numex:extopo}.
While the setups in Section \ref{s:numex:ex1} and \ref{s:numex:ex2} consider smooth domains with a smooth evolution and are supposed to validate the theoretical prediction,
in Section \ref{s:mass} we address concerns related to the conservation properties of the method.
Finally, in Section \ref{s:numex:extopo} we consider a problem with topology changes that demonstrates the usability of the method beyond the theoretical assumptions made.

\subsection{General setup} \label{s:numex:setup}
We discuss the geometry approximation that is used in all numerical experiments, discuss the discretization parameters that are varied and the quantities of interest that are measured during the numerical studies.

\subsubsection{Domain approximation with level sets}
In the numerical experiments we use a level set description of the domains $\Om{n}{}$, i.e. we assume that we are given a level set function $\phi: \wOm \rightarrow \rr$ so that $\Om{n}{} = \{ \phi(t_n) < 0 \}$. For the approximation of the domain we use an interpolation of $\phi(t_n)$, denoted by $\phi_h^n$.
Thereby, we define the approximate geometries as $\Om{n}{h} = \{\phi_h^n < 0 \}$.
Further, we use a level set function which is an approximate signed distance function so that we can use $\phi_h^n$ to make sense of the strip domains $\Sdh{n}$ which are then replaced by $\{ |\phi_h^n| \leq \delta_h \}$ defining $\Td{n}$ and $\Fh^{n}$.
In all numerical examples we consider a piecewise linear approximation of the geometry, i.e. $q=1$.

\subsubsection{Implementational details}
In all experiments we use the bilinear form
\begin{equation}
 \ahn(u_h,v_h) =
 \int_{\Om{n}{h}} \alpha \nabla u_h\cdot \nabla v_h\, dx
 + \int_{\Om{n}{h}} (\bw^e\cdot\nabla u_h) \, v_h
 + \int_{\Om{n}{h}} \Div(\bw^e) u_h v_h \, dx,
\end{equation}
which is slightly different compared to the bilinear form in \eqref{e:anh} that has been used in the analysis. Further, we introduce a right hand side source term $f(v_h) := \int_{\Omega_h^n} f v_h \, dx$ to \eqref{e:unfFEM1} that we use in some of the examples.

All implementations are done in \texttt{ngsxfem}\cite{ngsxfem}, an Add-On package for unfitted finite elements in the general purpose finite element solver \texttt{Netgen}/\texttt{NGSolve} \cite{netgen,ngsolve}.

\subsubsection{Discretization parameters}
In every of the following examples we consider \emph{unstructured} triangular, quasi-uniform meshes with an initial mesh size which we denote by $h_0$ and consider an initial time step size $\Delta t_0$. Starting from here, we apply successive uniform refinements in space and in time and denote the corresponding space and time refinement levels as $L_x = 0, \dots$ and $L_t = 0, \dots$, s.t. $h = h_0 \cdot 2^{-L_x}$ and  $\Delta t = \Delta t_0 \cdot 2^{-L_t}$.
For the stabilization we choose $s_h^n(\cdot,\cdot) = s_h^{n,\text{dir}}(\cdot,\cdot)$ and $\gamma_s = c_\gamma \tilde{K}$ where for $\tilde{K}$ we choose $\tilde{K} = \lceil \delta_h / h \rceil$ which is an estimate of the thickness of the stabilization strip in terms of elements and there holds $\tilde{K} \simeq K \simeq 1 + \delta_h/h$.
Here, $\delta_h = \winfn \Delta t$ where we exploit that we know $\winfn$ explicitly in all following examples.
For $c_\gamma$ we choose $c_\gamma = 1$ if we do not address the parameter otherwise.
To check for the influence of the stabilization, we consider different choices for $\gamma_s$ in Section \ref{sec:diffstab}.
The time steps in all the numerical experiments are chosen such that the (in these cases mild) condition \eqref{eq:xih} is fulfilled.

In all numerical examples we consider polynomial degree $m=1$.
We consider the use of the implicit Euler method primarily treated in this work, but also of a BDF2 discretization, cf. Remark \ref{rem:bdf2}.

\subsubsection{Quantities of interest}
We consider the errors in the discrete space-time norms
\begin{equation}
  \begin{split}
  \norm{u_h - u^e}_{L^2(L^2)}^2 &:= \sum_{n=1}^{N} \Delta t \norm{ u_h - u^e}_{\Om{n}{h}}^2 , \quad
  \norm{u_h - u^e}_{L^2(H^1)}^2 := \sum_{n=1}^{N} \Delta t \norm{ \nabla (u_h - u^e) }_{\Om{n}{h}}^2, \\
  \norm{u_h - u^e}_{L^\infty(L^2)} & := \max_{n=1,..,N} \norm{ u_h - u^e }_{\Om{n}{h}},
  \end{split}
\end{equation}
which we also denote as the $L^2(L^2)$, the $L^2(H^1)$ and $L^\infty(L^2)$ error, respectively. We notice that we have the a priori error estimate $\norm{u_h - u^e}_{L^2(H^1)} \lesssim (h^m + \Delta t) \cdot (1 + \frac{\Delta t}{h})^{\frac12}$ from the error analysis of the implicit Euler method. For the BDF2 scheme we expect an improved rate in time and we expect (without theoretical justification yet) also an additional order in space for the norm $\norm{u_h - u^e}_{L^2(L^2)}$.

To display the asymptotical convergence rates in space and time, we use the ``experimental order of convergence''(eoc) in space and time ( \eoc{x} / \eoc{t}) which is computed based on two errors of successive levels. Additionally, we compute the \eoc{} for combined refinements in space and time (\eoc{xt}). For the case of the $L^2(L^2)$ norm where we expect a convergence rate $(\Delta t + h^2)\cdot (1 + \frac{\Delta t}{h})^{\frac12}$ for the implicit Euler method, we also add the \eoc{} of combined refinement in space and time where to each level of refinement in space we use \emph{two} levels of refinements in time (\eoc{xtt}).

\subsection{Example 1: Traveling circle}\label{s:numex:ex1}
As a first example we consider a circle traveling with a time-dependent velocity field that is constant in space through a background mesh. The setup is taken from \cite{preussmaster}.
\subsubsection{Setup}
We fix the background domain to $\wOm = (-0.7,0.9) \times (-0.7,0.7)$ and consider the time interval $[0,T]$ with $T=0.2$.
The geometry evolution is based on the following functions,
\begin{equation*}
  \phi(\bx,t) = \Vert \bx - \rho(\bx,t) \Vert - R_0, \quad
  \rho(\bx,t) = (1/\pi \sin(2 \pi t),0)^T, \quad
  \bw(\bx,t) = \bw^e(\bx,t) = \partial_t \rho(\bx,t), \quad R_0=0.5.
\end{equation*}
We set $\alpha=1$ and the (extended) solution is given as
\begin{equation*}
  u^e(\bx,t) = \cos^2 \left( \frac{\pi}{2 R} \norm{\bx - \rho(\bx,t)}_2\right),
\end{equation*}
which fulfills the boundary conditions due to $(- \nabla u^e \cdot \bn)|_{\Gamma(t)} = 0,~t\in[0,T]$. We choose $f$ according to $u^e$.
Initial temporal and spatial resolution are chosen as $h_0=0.2$, $\Delta t=0.1$.

\begin{table}
  \begin{center}
    \footnotesize
    \begin{tabular}{lrrrrrrrrc}
\toprule
\refts    &                 0&                 1&                 2&                 3&                 4&                 5&                 6&                 7&        \eoc{t}      \\
\midrule
   0&\nbf{1.574589e-01}&\num{9.734332e-02}&\num{6.101333e-02}&\num{4.224456e-02}&\num{3.379677e-02}&\num{3.034990e-02}&\num{2.899458e-02}&\num{2.845565e-02}&        ---       \\
   1&\num{1.425875e-01}&\nbf{8.523002e-02}&\num{4.837002e-02}&\num{2.906107e-02}&\num{2.030774e-02}&\num{1.691725e-02}&\num{1.575183e-02}&\num{1.537054e-02}&\numQ{      0.889}\\
   2&\num{1.405987e-01}&\num{7.754668e-02}&\nbf{4.284082e-02}&\num{2.321583e-02}&\num{1.385749e-02}&\num{9.960733e-03}&\num{8.601672e-03}&\num{8.189170e-03}&\numQ{      0.908}\\
   3&\num{1.399389e-01}&\num{7.667822e-02}&\num{3.975617e-02}&\nbf{2.099506e-02}&\num{1.126176e-02}&\num{6.784928e-03}&\num{5.002498e-03}&\num{4.413833e-03}&\numQ{      0.892}\\
   4&\num{1.396726e-01}&\num{7.626979e-02}&\num{3.937225e-02}&\num{1.994142e-02}&\nbf{1.034021e-02}&\num{5.541579e-03}&\num{3.370015e-03}&\num{2.524764e-03}&\numQ{      0.806}\\
   5&\num{1.395826e-01}&\num{7.609561e-02}&\num{3.920565e-02}&\num{1.979783e-02}&\num{9.975743e-03}&\nbf{5.128487e-03}&\num{2.749836e-03}&\num{1.682293e-03}&\numQ{      0.586}\\
   6&\num{1.395376e-01}&\num{7.599753e-02}&\num{3.912465e-02}&\num{1.974102e-02}&\num{9.914052e-03}&\num{4.987553e-03}&\nbf{2.553543e-03}&\num{1.369962e-03}&\numQ{      0.296}\\
   7&\num{1.395230e-01}&\num{7.595378e-02}&\num{3.908600e-02}&\num{1.971663e-02}&\num{9.893222e-03}&\num{4.959584e-03}&\num{2.493576e-03}&\nbf{1.274115e-03}&\numQ{      0.105}\\
\midrule
\eoc{x}&       ---          &\numQ{      0.877}&\numQ{      0.958}&\numQ{      0.987}&\numQ{      0.995}&\numQ{      0.996}&\numQ{      0.992}&\numQ{      0.969}&\\
\midrule
\underline{\eoc{xt}}&      ---            &\numQ{      0.886}&\numQ{      0.992}&\numQ{      1.029}&\numQ{      1.022}&\numQ{      1.012}&\numQ{      1.006}&\numQ{      1.003}&\\
\bottomrule
  \end{tabular}
\end{center}
\caption{ $L^2(H^1)$ error for the implicit Euler method for Example 1.} \vspace*{-0.5cm}
\label{tab:ex1:l2h1:ie}
\end{table}

\begin{table}
  \begin{center}
    \footnotesize
    \begin{tabular}{lrrrrrrrrc}
\toprule
\refts    &                 0&                 1&                 2&                 3&                 4&                 5&                 6&                 7&        \eoc{t}      \\
\midrule
   0&\nbf{1.865501e-02}&\num{8.970484e-03}&\num{6.096687e-03}&\num{5.129064e-03}&\num{4.777110e-03}&\num{4.635451e-03}&\num{4.574100e-03}&\num{4.546177e-03}&        ---       \\
   1&\num{1.422632e-02}&\nbf{5.590127e-03}&\num{3.317888e-03}&\num{2.719423e-03}&\nit{2.523780e-03}&\num{2.447789e-03}&\num{2.415146e-03}&\num{2.399942e-03}&\numQ{      0.922}\\
   2&\num{1.329629e-02}&\num{3.960299e-03}&\nbf{1.883419e-03}&\num{1.427349e-03}&\num{1.316174e-03}&\num{1.280232e-03}&\num{1.266720e-03}&\num{1.261032e-03}&\numQ{      0.928}\\
   3&\num{1.304036e-02}&\num{3.572794e-03}&\num{1.204125e-03}&\nbf{7.586096e-04}&\num{6.756068e-04}&\nit{6.566022e-04}&\num{6.515136e-04}&\num{6.501298e-04}&\numQ{      0.956}\\
   4&\num{1.298553e-02}&\num{3.470573e-03}&\num{9.904617e-04}&\num{4.399294e-04}&\nbf{3.513298e-04}&\num{3.355212e-04}&\num{3.323378e-04}&\num{3.318267e-04}&\numQ{      0.970}\\
   5&\num{1.297840e-02}&\num{3.447442e-03}&\num{9.154310e-04}&\num{3.063069e-04}&\num{1.905389e-04}&\nbf{1.720831e-04}&\nit{1.687564e-04}&\num{1.681832e-04}&\numQ{      0.980}\\
   6&\num{1.297417e-02}&\num{3.440591e-03}&\num{8.875598e-04}&\num{2.525991e-04}&\num{1.134425e-04}&\num{8.973066e-05}&\nbf{8.562924e-05}&\num{8.487192e-05}&\numQ{      0.987}\\
   7&\num{1.297422e-02}&\num{3.439324e-03}&\num{8.763359e-04}&\num{2.314029e-04}&\num{7.885197e-05}&\num{4.869846e-05}&\num{4.373282e-05}&\nit{4.277358e-05}&\numQ{      0.989}\\
\midrule
\eoc{x}&       ---       &\numQ{      1.915}&\numQ{      1.973}&\numQ{      1.921}&\numQ{      1.553}&\numQ{      0.695}&\numQ{      0.155}&\numQ{      0.032}&\\
\midrule
\underline{\eoc{xt}}&      ---        &\numQ{      1.739}&\numQ{      1.570}&\numQ{      1.312}&\numQ{      1.111}&\numQ{      1.030}&\numQ{      1.007}&\numQ{      1.001}&\\
\midrule
\underline{\underline{\eoc{xtt}}}&     ---         &---&---&---&---&\numQ{      1.95}&\numQ{      1.96}&\numQ{      1.98}&\\
\bottomrule
  \end{tabular}
\end{center}
\caption{ $L^2(L^2)$ error for the implicit Euler method for Example 1.} \vspace*{-0.5cm}
\label{tab:ex1:l2l2:ie}
\end{table}

\subsubsection{Convergence in space and time}
In Table \ref{tab:ex1:l2h1:ie} and Table \ref{tab:ex1:l2l2:ie} we display the $L^2(H^1)$ norm and the $L^2(L^2)$ norm for $8$ different time and space levels and corresponding \eoc{}s for the implicit Euler method. We observe the convergence behavior $\norm{u_h - u^e}_{L^2(H^1)} \lesssim (h + \Delta t)$. This is better than predicted as we do not observe the influence of the anisotropy factor $K \simeq 1 + \delta_h/h$. Below, in the other experiments we also do not see a significant impact of this factor on the results.
In the $L^2(L^2)$ norm we observe the improved convergence rate
$\norm{u_h - u^e}_{L^2(L^2)} \lesssim (h^2 + \Delta t)$.
In Table \ref{tab:ex1:l2l2:bdf2} we consider the $L^2(L^2)$ norm error of a BDF2 discretization and observe
$\norm{u_h - u^e}_{L^2(L^2)} \lesssim (h^2 + \Delta t^2)$.

In Table \ref{tab:ex1:linftyl2:ie} and Table \ref{tab:ex1:linftyl2:bdf2} we also show the convergence of the implicit Euler and the BDF2 method in the $L^\infty(L^2)$ norm. We observe the same convergence rates as in the $L^2(L^2)$ norm.

\begin{table}
  \begin{center}
    \footnotesize
    \begin{tabular}{lrrrrrrrrc}
\toprule
\refts    &                 0&                 1&                 2&                 3&                 4&                 5&                 6&                 7&        \eoc{t}      \\
\midrule
   0&\nbf{2.160706e-02}&\num{1.183087e-02}&\num{7.640618e-03}&\num{5.974943e-03}&\num{5.298223e-03}&\num{5.015998e-03}&\num{4.886525e-03}&\num{4.820566e-03}&        ---       \\
   1&\num{1.612822e-02}&\nbf{6.336613e-03}&\num{3.637545e-03}&\num{2.745107e-03}&\num{2.372000e-03}&\num{2.193798e-03}&\num{2.100457e-03}&\num{2.047594e-03}&\numQ{      1.235}\\
   2&\num{1.400166e-02}&\num{4.629212e-03}&\nbf{1.711648e-03}&\num{9.146368e-04}&\num{7.022108e-04}&\num{6.365059e-04}&\num{6.131291e-04}&\num{6.037330e-04}&\numQ{      1.762}\\
   3&\num{1.357721e-02}&\num{3.767018e-03}&\num{1.150143e-03}&\nbf{3.903920e-04}&\num{2.065529e-04}&\num{1.705793e-04}&\num{1.655256e-04}&\num{1.660907e-04}&\numQ{      1.862}\\
   4&\num{1.329448e-02}&\num{3.602366e-03}&\num{9.429308e-04}&\num{2.688762e-04}&\nbf{9.129648e-05}&\num{5.205046e-05}&\num{4.607950e-05}&\num{4.593658e-05}&\numQ{      1.854}\\
   5&\num{1.314612e-02}&\num{3.521127e-03}&\num{9.027692e-04}&\num{2.270193e-04}&\num{6.348785e-05}&\nbf{2.218179e-05}&\num{1.359735e-05}&\num{1.244194e-05}&\numQ{      1.884}\\
   6&\num{1.306436e-02}&\num{3.479799e-03}&\num{8.839244e-04}&\num{2.202702e-04}&\num{5.559105e-05}&\num{1.534048e-05}&\nbf{5.486552e-06}&\num{3.515896e-06}&\numQ{      1.823}\\
   7&\num{1.301879e-02}&\num{3.459705e-03}&\num{8.751645e-04}&\num{2.176504e-04}&\num{5.429783e-05}&\num{1.373355e-05}&\num{3.763809e-06}&\nbf{1.366282e-06}&\numQ{      1.364}\\
\midrule
\eoc{x}&       ---        &\numQ{      1.912}&\numQ{      1.983}&\numQ{      2.008}&\numQ{      2.003}&\numQ{      1.983}&\numQ{      1.867}&\numQ{      1.462}&\\
\midrule
\underline{\eoc{xt}}&     ---         &\numQ{      1.770}&\numQ{      1.888}&\numQ{      2.132}&\numQ{      2.096}&\numQ{      2.041}&\numQ{      2.015}&\numQ{      2.006}&\\
\bottomrule
  \end{tabular}
\end{center}
\caption{ $L^2(L^2)$ error for the BDF2 method for Example 1.} \vspace*{-0.5cm}
\label{tab:ex1:l2l2:bdf2}
\end{table}

\begin{table}
  \begin{center}
    \footnotesize
    \begin{tabular}{lrrrrrrrrc}
\toprule
\refts    &                 0&                 1&                 2&                 3&                 4&                 5&                 6&                 7&        \eoc{t}      \\
\midrule
   0&\nbf{9.729384e-02}&\num{4.933436e-02}&\num{3.612165e-02}&\num{3.125511e-02}&\num{2.935682e-02}&\num{2.854270e-02}&\num{2.816718e-02}&\num{2.798604e-02}&        ---       \\
   1&\num{8.499111e-02}&\nbf{2.961477e-02}&\num{1.775635e-02}&\num{1.459882e-02}&\nit{1.358984e-02}&\num{1.327500e-02}&\num{1.322694e-02}&\num{1.323343e-02}&\numQ{      1.081}\\
   2&\num{8.560402e-02}&\num{2.104107e-02}&\nbf{1.066020e-02}&\num{8.046866e-03}&\num{7.508456e-03}&\num{7.409310e-03}&\num{7.406471e-03}&\num{7.422717e-03}&\numQ{      0.834}\\
   3&\num{8.906167e-02}&\num{2.069431e-02}&\num{6.847485e-03}&\nbf{4.494342e-03}&\num{4.029314e-03}&\nit{3.954099e-03}&\num{3.953388e-03}&\num{3.964781e-03}&\numQ{      0.905}\\
   4&\num{9.151676e-02}&\num{2.168527e-02}&\num{5.398656e-03}&\num{2.655553e-03}&\nbf{2.176031e-03}&\num{2.090345e-03}&\num{2.079830e-03}&\num{2.083057e-03}&\numQ{      0.929}\\
   5&\num{9.294808e-02}&\num{2.255565e-02}&\num{5.579065e-03}&\num{1.753575e-03}&\num{1.194318e-03}&\nbf{1.090903e-03}&\nit{1.073015e-03}&\num{1.071417e-03}&\numQ{      0.959}\\
   6&\num{9.369835e-02}&\num{2.306068e-02}&\num{5.774145e-03}&\num{1.447288e-03}&\num{6.912667e-04}&\num{5.724796e-04}&\nbf{5.487490e-04}&\num{5.446962e-04}&\numQ{      0.976}\\
   7&\num{9.408949e-02}&\num{2.332195e-02}&\num{5.887350e-03}&\num{1.457718e-03}&\num{4.483664e-04}&\num{3.077039e-04}&\num{2.812159e-04}&\nit{2.755613e-04}&\numQ{      0.983}\\
\midrule
\eoc{x}&       ---        &\numQ{      2.012}&\numQ{      1.986}&\numQ{      2.014}&\numQ{      1.701}&\numQ{      0.543}&\numQ{      0.130}&\numQ{      0.029}&\\
\midrule
\underline{\eoc{xt}}&        ---       &\numQ{      1.716}&\numQ{      1.474}&\numQ{      1.246}&\numQ{      1.046}&\numQ{      0.996}&\numQ{      0.991}&\numQ{      0.994}&\\
\midrule
\underline{\underline{\eoc{xtt}}}&        ---       & --- & --- & --- & --- &\numQ{      1.781}&\numQ{      1.882}&\numQ{      1.961}&\\
\bottomrule
  \end{tabular}
\end{center}
\caption{ $L^\infty(L^2)$ error for the implicit Euler method for Example 1.} \vspace*{-0.5cm}
\label{tab:ex1:linftyl2:ie}
\end{table}

\begin{table}
  \begin{center}
    \footnotesize
    \begin{tabular}{lrrrrrrrrc}
\toprule
\refts    &                 0&                 1&                 2&                 3&                 4&                 5&                 6&                 7&        \eoc{t}      \\
\midrule
   0&\nbf{1.087759e-01}&\num{6.641078e-02}&\num{4.551070e-02}&\num{3.617021e-02}&\num{3.181873e-02}&\num{2.977182e-02}&\num{2.878297e-02}&\num{2.829526e-02}&        ---       \\
   1&\num{9.371222e-02}&\nbf{3.436691e-02}&\num{2.153336e-02}&\num{1.660362e-02}&\num{1.460556e-02}&\num{1.373658e-02}&\num{1.333372e-02}&\num{1.313881e-02}&\numQ{      1.107}\\
   2&\num{9.555405e-02}&\num{2.575611e-02}&\nbf{9.600374e-03}&\num{6.638909e-03}&\num{5.679017e-03}&\num{5.339504e-03}&\num{5.205055e-03}&\num{5.145946e-03}&\numQ{      1.352}\\
   3&\num{9.513787e-02}&\num{2.398256e-02}&\num{6.434050e-03}&\nbf{2.519013e-03}&\num{1.972939e-03}&\num{1.836289e-03}&\num{1.793526e-03}&\num{1.777901e-03}&\numQ{      1.533}\\
   4&\num{9.489700e-02}&\num{2.371567e-02}&\num{6.115137e-03}&\num{1.612267e-03}&\nbf{7.163662e-04}&\num{6.216198e-04}&\num{6.017498e-04}&\num{5.978225e-04}&\numQ{      1.572}\\
   5&\num{9.477414e-02}&\num{2.360406e-02}&\num{6.061370e-03}&\num{1.530559e-03}&\num{4.011432e-04}&\nbf{2.026121e-04}&\num{1.895190e-04}&\num{1.870506e-04}&\numQ{      1.676}\\
   6&\num{9.467728e-02}&\num{2.361989e-02}&\num{6.045541e-03}&\num{1.519677e-03}&\num{3.820443e-04}&\num{9.962677e-05}&\nbf{5.654871e-05}&\num{5.422225e-05}&\numQ{      1.786}\\
   7&\num{9.456274e-02}&\num{2.360911e-02}&\num{6.037779e-03}&\num{1.515790e-03}&\num{3.793685e-04}&\num{9.523443e-05}&\num{2.479995e-05}&\nbf{1.517557e-05}&\numQ{      1.837}\\
\midrule
\eoc{x}:&     ---          &\numQ{      2.002}&\numQ{      1.967}&\numQ{      1.994}&\numQ{      1.998}&\numQ{      1.994}&\numQ{      1.941}&\numQ{      0.709}&\\
\midrule
\underline{\eoc{xt}}&     ---          &\numQ{      1.662}&\numQ{      1.840}&\numQ{      1.930}&\numQ{      1.814}&\numQ{      1.822}&\numQ{      1.841}&\numQ{      1.898}&\\
\bottomrule
  \end{tabular}
\end{center}
\caption{ $L^\infty(L^2)$ error for the BDF2 method for Example 1.} \vspace*{-0.5cm}
\label{tab:ex1:linftyl2:bdf2}
\end{table}

\subsubsection{Influence of the stabilization parameter} \label{sec:diffstab}
Next, we are interested in the sensitivity of the error on the choice of the stabilization scaling $\gamma_s$. To this end, we fix $L_t=3$ and we vary $L_x \in \{0,\dots,7\}$ and $c_\gamma \in \{ 1/100, 1 , 100\}$. Further, we consider two cases: First, we use the scaling of $\gamma_s$ with $\tilde{K}$, the thickness of the extension strip and secondly, a constant scaling of $\gamma_s$, i.e.~$\gamma_s = c_\gamma$.
\begin{table}
  \begin{center}
    \footnotesize
    \begin{tabular}{lrrrrrrrrc}
\toprule
 $L_x$    &                 0&                 1&                 2&                 3&                 4&                 5&                 6&                 7 \\
\midrule
      $\tilde{K}$ & 1 & 1 & 1 & 2 & 3 & 5 & 9 & 17 \\
\midrule
      $\gamma_s = 0.01 \cdot \tilde{K}$
      &\num{1.208728e-02}&\num{3.313700e-03}&\num{1.155118e-03}&\num{7.320646e-04}&\num{6.663782e-04}&\num{6.537001e-04}&\num{6.508210e-04}&\num{6.501663e-04}\\
      $\gamma_s = \tilde{K}$
      &\num{1.304036e-02}&\num{3.572794e-03}&\num{1.204125e-03}&\num{7.586096e-04}&\num{6.756068e-04}&\num{6.566022e-04}&\num{6.515136e-04}&\num{6.501298e-04}\\
      $\gamma_s = 100 \cdot \tilde{K}$
      &\num{3.450388e-02}&\num{8.428707e-03}&\num{2.169383e-03}&\num{9.899897e-04}&\num{7.600030e-04}&\num{7.035589e-04}&\num{6.834249e-04}&\num{6.735384e-04}\\
      $\gamma_s = 0.01$
      &\num{1.208728e-02}&\num{3.313700e-03}&\num{1.155118e-03}&\num{7.322809e-04}&\num{6.664398e-04}&\num{6.537602e-04}&\num{6.508758e-04}&\num{6.502156e-04}\\
      $\gamma_s = 1$
      &\num{1.304036e-02}&\num{3.572794e-03}&\num{1.204125e-03}&\num{7.425655e-04}&\num{6.685492e-04}&\num{6.538806e-04}&\num{6.507143e-04}&\num{6.500822e-04}\\
      $\gamma_s = 100$
      &\num{3.450388e-02}&\num{8.428707e-03}&\num{2.169383e-03}&\num{9.694448e-04}&\num{7.532647e-04}&\num{6.914811e-04}&\num{6.646766e-04}&\num{6.531096e-04}\\
      \bottomrule
  \end{tabular}
\end{center}
\caption{ $L^2(L^2)$ error for implicit Euler method in Example 1 for different stabilization scalings $\gamma_s$, $L_x=0,\dots,7$, $L_t=3$.} \vspace*{-0.5cm}
\label{tab:ex1:l2l2:bdf2:stab}
\end{table}

In Table \ref{tab:ex1:l2l2:bdf2:stab} the results for a fixed time resolution are shown. We observe that there is only a very mild dependency of the numerical results on the choice of the stabilization parameter. Further, we observe that the anisotropy scaling with $\tilde{K}$ seems not to have a significant effect.

\subsection{Example 2: Growing / shrinking circle}\label{s:numex:ex2}
In this example we consider growing and shrinking circles which are described below in similar setups.
\subsubsection{Setups}
We fix the background domain to be $\wOm = (-1.25,1.25) \times (-1.25,1.25)$ and fix the time interval to $[0,T],~T=\operatorname{ln}(2)$.
The geometry evolution for the growing circle is based on the following functions:
\begin{equation*}
  \phi(\bx,t) = \Vert \bx \Vert - R(t), \quad
  R(t) = R_0 e^t,~R_0 = \frac12, \quad
  \bw(\bx,t) = \bx.
\end{equation*}
For the shrinking sphere we take accordingly
\begin{equation*}
  \phi(\bx,t) = \Vert \bx \Vert - R(t), \quad
  R(t) = R_0 e^{-t},~R_0 = 1, \quad
  \bw(\bx,t) = -\bx.
\end{equation*}
This corresponds to a circle growing from radius $\frac12$ to radius $1$ for the one case and a circle shrinking from radius $1$ to radius $\frac12$ for the other.
We obtain the constants
$  \winfn = 1$,
$\Div(\bw) = \pm 2$.
We choose the diffusivity $\alpha = 0.2$ and the right hand side $f$ so that the manufactured solution is (in both cases)
\begin{equation*}
  u^e(\bx,t) = \cos( \pi r / R(t) ),
\end{equation*}
which fulfills the boundary conditions. As initial resolution we choose $h=0.4$, $\Delta t=0.5$.

\begin{table}
  \begin{center}
    \footnotesize
    \begin{tabular}{lrrrrrrrrc}
\toprule
\refts    &                 0&                 1&                 2&                 3&                 4&                 5&                 6&                 7&       \eoc{t}      \\
\midrule
   0&\nbf{3.006714e-01}&\num{2.193849e-01}&\num{1.626635e-01}&\num{1.360271e-01}&\num{1.195835e-01}&\num{1.088052e-01}&\num{1.011408e-01}&\num{9.560677e-02}&        ---       \\
   1&\num{2.462599e-01}&\nbf{1.342448e-01}&\num{8.018081e-02}&\num{4.945041e-02}&\num{3.492584e-02}&\num{2.630218e-02}&\num{2.132810e-02}&\num{1.847133e-02}&\numQ{      2.372}\\
   2&\num{2.458792e-01}&\num{8.862904e-02}&\nbf{4.322672e-02}&\num{2.386403e-02}&\num{1.364277e-02}&\num{9.065540e-03}&\num{6.678332e-03}&\num{5.615297e-03}&\numQ{      1.718}\\
   3&\num{2.349839e-01}&\num{7.777465e-02}&\num{2.447670e-02}&\nbf{1.090908e-02}&\num{5.903193e-03}&\num{3.549645e-03}&\num{2.589624e-03}&\num{2.142131e-03}&\numQ{      1.390}\\
   4&\num{2.262955e-01}&\num{7.009652e-02}&\num{1.970686e-02}&\num{5.690150e-03}&\nbf{2.379511e-03}&\num{1.286580e-03}&\num{8.250290e-04}&\num{6.532190e-04}&\numQ{      1.713}\\
   5&\num{2.213395e-01}&\num{6.555979e-02}&\num{1.706490e-02}&\num{4.434785e-03}&\num{1.239209e-03}&\nbf{4.966944e-04}&\num{2.719942e-04}&\num{1.881381e-04}&\numQ{      1.796}\\
   6&\num{2.184885e-01}&\num{6.330986e-02}&\num{1.594460e-02}&\num{3.909461e-03}&\num{9.893903e-04}&\num{2.717112e-04}&\nbf{1.049087e-04}&\num{5.878464e-05}&\numQ{      1.678}\\
   7&\num{2.168136e-01}&\num{6.221968e-02}&\num{1.542897e-02}&\num{3.690839e-03}&\num{8.971931e-04}&\num{2.258335e-04}&\num{6.167212e-05}&\nbf{2.307751e-05}&\numQ{      1.349}\\
\midrule
\eoc{x}&        ---      &\numQ{      1.801}&\numQ{      2.012}&\numQ{      2.064}&\numQ{      2.040}&\numQ{      1.990}&\numQ{      1.873}&\numQ{      1.418}&\\
\midrule
\eoc{xt}&        ---       &\numQ{      1.163}&\numQ{      1.635}&\numQ{      1.986}&\numQ{      2.197}&\numQ{      2.260}&\numQ{      2.243}&\numQ{      2.185}&\\
\bottomrule
  \end{tabular}
\end{center}
\caption{ $L^2(L^2)$ error for the BDF2 method for the growing circle in Example 2.} \vspace*{-0.5cm}
\label{tab:ex2a:l2l2:bdf2}
\end{table}

\begin{table}
  \begin{center}
    \footnotesize
    \begin{tabular}{lrrrrrrrrc}
\toprule
\refts   &                 0&                 1&                 2&                 3&                 4&                 5&                 6&                 7&        \eoc{t}      \\
\midrule
   0&\nbf{1.283044e+00}&\num{8.164588e-01}&\num{5.031516e-01}&\num{3.832734e-01}&\num{3.508497e-01}&\num{3.589056e-01}&\num{3.781116e-01}&\num{3.950260e-01}&        ---       \\
   1&\num{3.124208e-01}&\nbf{1.667540e-01}&\num{1.032282e-01}&\num{7.754932e-02}&\num{7.849631e-02}&\num{8.681905e-02}&\num{9.482725e-02}&\num{1.005935e-01}&\numQ{      1.973}\\
   2&\num{2.271355e-01}&\num{8.835345e-02}&\nbf{4.475136e-02}&\num{2.817294e-02}&\num{2.450832e-02}&\num{2.598077e-02}&\num{2.816251e-02}&\num{2.992091e-02}&\numQ{      1.749}\\
   3&\num{2.070014e-01}&\num{7.234825e-02}&\num{2.341306e-02}&\nbf{1.105665e-02}&\num{7.584713e-03}&\num{7.273642e-03}&\num{7.761841e-03}&\num{8.303891e-03}&\numQ{      1.849}\\
   4&\num{1.962127e-01}&\num{6.389730e-02}&\num{1.825899e-02}&\num{5.443160e-03}&\nbf{2.591802e-03}&\num{1.962660e-03}&\num{1.987877e-03}&\num{2.127564e-03}&\numQ{      1.965}\\
   5&\num{1.906990e-01}&\num{5.952491e-02}&\num{1.581510e-02}&\num{4.149547e-03}&\num{1.221848e-03}&\nbf{6.062333e-04}&\num{4.972845e-04}&\num{5.191233e-04}&\numQ{      2.035}\\
   6&\num{1.875775e-01}&\num{5.744472e-02}&\num{1.478114e-02}&\num{3.652240e-03}&\num{9.327261e-04}&\num{2.778329e-04}&\nbf{1.440025e-04}&\num{1.248721e-04}&\numQ{      2.056}\\
   7&\num{1.858838e-01}&\num{5.640264e-02}&\num{1.429550e-02}&\num{3.448303e-03}&\num{8.405487e-04}&\num{2.144175e-04}&\num{6.508746e-05}&\nbf{3.487051e-05}&\numQ{      1.840}\\
\midrule
\eoc{x}&       ---        &\numQ{      1.721}&\numQ{      1.980}&\numQ{      2.052}&\numQ{      2.036}&\numQ{      1.971}&\numQ{      1.720}&\numQ{      0.900}&\\
\midrule
\eoc{xt}&      ---         &\numQ{      2.944}&\numQ{      1.898}&\numQ{      2.017}&\numQ{      2.093}&\numQ{      2.096}&\numQ{      2.074}&\numQ{      2.046}&\\
\bottomrule
  \end{tabular}
\end{center}
\caption{ $L^2(L^2)$ error for the BDF2 method for the shrinking circle in Example 2.} \vspace*{-0.5cm}
\label{tab:ex2b:l2l2:bdf2}
\end{table}

\subsubsection{Convergence in space and time}
In the Tables \ref{tab:ex2a:l2l2:bdf2} and \ref{tab:ex2b:l2l2:bdf2} the $L^2(L^2)$ errors for the examples are shown using the BDF2 method. We again observe an error behavior of the form $h^2 + \Delta t^2$. The results are in agreement with the previous observations.

\subsection{Example 3: Mass conservation}\label{s:mass}
In Section \ref{s:form} we give a conservation property that is fulfilled by the exact solution to the problem. However, due to the fact that we rely on a discrete extension, we do not preserve this property on the discrete level. In this final example we want to investigate the mass loss for the geometrical setup of the first example. We set $u^0 = \sin(\pi \Vert \bx - \rho(\bx,t)\Vert_2)$ and $\alpha = 0.1$.
This time we consider $f = 0$ so that the total mass of the exact solution, i.e. $U = U^k = \int_{\Om{k}{}} u^k \, dx, k=1,..,N$, is constant over time. Analogously we define the discrete mass
$U_h^k = \int_{\Om{k}{h}} u_h^k \, dx, k=1,..,N$.
In Table \ref{tab:ex3:mass:bdf2} we display the maximum deviation from the inital mass $E_h^{\text{mass}} = \max_{k=1,..,N} | U_h^k - U_h^0 |$.
\begin{table}
  \begin{center}
    \footnotesize
    \begin{tabular}{lrrrrrrrrc}
\toprule
    &                 0&                 1&                 2&                 3&                 4&                 5&                 6&                 7&        eoc.      \\
\midrule
   0&\nbf{3.581696e-02}&\num{2.568997e-02}&\num{1.774261e-02}&\num{1.405741e-02}&\num{1.366993e-02}&\num{1.467462e-02}&\num{1.585093e-02}&\num{1.675772e-02}&        ---       \\
   1&\num{3.871888e-03}&\nbf{6.171417e-03}&\num{5.251590e-03}&\num{3.778913e-03}&\num{2.384315e-03}&\num{1.287903e-03}&\num{1.832939e-03}&\num{2.175918e-03}&\numQ{      2.945}\\
   2&\num{2.196668e-03}&\num{2.435972e-03}&\nbf{1.764942e-03}&\num{1.253959e-03}&\num{8.414086e-04}&\num{5.130043e-04}&\num{2.824537e-04}&\num{2.555932e-04}&\numQ{      3.090}\\
   3&\num{5.856062e-04}&\num{6.830772e-04}&\num{4.052278e-04}&\nbf{2.872539e-04}&\num{1.958217e-04}&\num{1.306721e-04}&\num{8.195939e-05}&\num{9.629641e-05}&\numQ{      1.408}\\
   4&\num{8.179372e-04}&\num{2.272683e-04}&\num{4.926680e-05}&\num{5.862617e-05}&\nbf{3.680309e-05}&\num{2.539949e-05}&\num{1.732297e-05}&\num{2.346089e-05}&\numQ{      2.037}\\
   5&\num{9.293895e-04}&\num{3.336667e-04}&\num{8.172195e-05}&\num{1.310765e-05}&\num{7.611021e-06}&\nbf{4.697839e-06}&\num{3.180337e-06}&\num{4.740087e-06}&\numQ{      2.307}\\
   6&\num{9.465101e-04}&\num{3.535768e-04}&\num{9.389850e-05}&\num{1.388631e-05}&\num{3.129111e-06}&\num{1.013360e-06}&\nbf{5.640137e-07}&\num{8.505017e-07}&\numQ{      2.479}\\
   7&\num{9.495255e-04}&\num{3.597228e-04}&\num{9.623942e-05}&\num{1.530132e-05}&\num{3.487484e-06}&\num{6.658895e-07}&\num{2.166283e-07}&\nbf{1.502422e-07}&\numQ{      2.501}\\
\midrule
eoc:&                  &\numQ{      1.400}&\numQ{      1.902}&\numQ{      2.653}&\numQ{      2.133}&\numQ{      2.389}&\numQ{      1.620}&\numQ{      0.528}&\\
\midrule
diag&                  &\numQ{      2.537}&\numQ{      1.806}&\numQ{      2.619}&\numQ{      2.964}&\numQ{      2.970}&\numQ{      3.058}&\numQ{      1.908}&\\
\bottomrule
  \end{tabular}
\end{center}
\caption{Max. deviation from conservation, $E_h^{\text{mass}}$, for BDF2 method in the example of Section \ref{s:mass}.} \vspace*{-0.5cm}
\label{tab:ex3:mass:bdf2}
\end{table}
We observe that this deviation is not zero, but converges with at least the same rate as the $L^2(L^2)$ norm.

To enforce a global constraint on the solution we use a Lagrange multiplier formulation as it has been done in \cite{kuhl1996constraint}. In the context of unfitted FEM the enforcement of exact global conservation has also been considered in \cite{hansbo2016cut}. This changes our formulation from \eqref{e:unfFEM1} to the following formulation: For a given $u_h^0 \in V_h^0$ find  $(u_h^n,\lambda) \in V_h^n\times \rr$, $n=1,\dots,N$, satisfying
\begin{equation}\label{e:unfFEM1:lag}
\int_{\Om{n}{h}} \frac{u_h^n-u_h^{n-1}}{\Delta t} v_h\, dx + \ahn(u^n_h, v_h)+ \gamma_s s_h^n(u_h^n,v_h) + \lambda \int_{\Om{n}{h}} v_h\, dx + \mu \int_{\Om{n}{h}} u_h^n\, dx  = \mu \int_{\Om{n-1}{h}} u_h^{n-1}\, dx \end{equation}
for all $v_h\in V_h^n,~\mu \in \rr$. Here $\lambda$ is the Lagrangian multiplier to the scalar constraint $\int_{\Om{n}{h}} u_h^n\, dx  = \int_{\Om{n-1}{h}} u_h^{n-1}\, dx$. The adaptation to the BDF2 scheme is straight-forward. We denote this solution as $u_h^{n,\ast}$. By construction this approach conserves mass exact. To demonstrate that it does not destroy the accuracy of the original method, we display the difference of the methods in the $L^2(\Om{N}{h})$ norm, $\Vert u_h^{N,\ast} - u_h^{N} \Vert_{\Om{N}{h}}$, in Table \ref{tab:ex3:diff:bdf2}.
We observe that the difference converges with the same order of convergence as the error of the original method as the norm of the difference tends to zero with the same order in space and time, although the \eoc{} is a bit less regular as in the experiments before. We conclude that the convergence properties of the original method is preserved, i.e. we can combine the method with a formation that ensures exact global conservation.

\begin{table}
  \begin{center}
    \footnotesize
    \begin{tabular}{lrrrrrrrrc}
\toprule
    &                 0&                 1&                 2&                 3&                 4&                 5&                 6&                 7&        eoc.      \\
\midrule
   0&\nbf{2.286573e-02}&\num{1.463576e-02}&\num{9.297817e-03}&\num{7.352705e-03}&\num{7.899919e-03}&\num{9.425932e-03}&\num{1.091032e-02}&\num{1.199930e-02}&        ---       \\
   1&\num{4.412119e-03}&\nbf{6.980113e-03}&\num{5.929228e-03}&\num{4.264685e-03}&\num{2.690514e-03}&\num{1.388115e-03}&\num{4.194123e-04}&\num{2.484140e-04}&\numQ{      5.594}\\
   2&\num{1.571063e-03}&\num{2.057683e-03}&\nbf{1.408740e-03}&\num{8.476178e-04}&\num{4.430137e-04}&\num{1.267910e-04}&\num{1.160217e-04}&\num{2.884063e-04}&\numQ{     -0.215}\\
   3&\num{2.385289e-05}&\num{3.758943e-04}&\num{2.449787e-04}&\nbf{1.257189e-04}&\num{4.678815e-05}&\num{1.386648e-05}&\num{6.629885e-05}&\num{1.086589e-04}&\numQ{      1.408}\\
   4&\num{3.567763e-04}&\num{3.223763e-05}&\num{3.355555e-05}&\num{2.124451e-05}&\nbf{5.587318e-06}&\num{5.009894e-06}&\num{1.598373e-05}&\num{2.647280e-05}&\numQ{      2.037}\\
   5&\num{4.556076e-04}&\num{5.688764e-05}&\num{3.462499e-06}&\num{2.727657e-06}&\num{2.611230e-07}&\nbf{9.495594e-07}&\num{3.036523e-06}&\num{5.348619e-06}&\numQ{      2.307}\\
   6&\num{4.687030e-04}&\num{7.276316e-05}&\num{4.769150e-06}&\num{3.800638e-07}&\num{1.165818e-06}&\num{2.257010e-07}&\nbf{5.217779e-07}&\num{9.596889e-07}&\numQ{      2.479}\\
   7&\num{4.707393e-04}&\num{7.786413e-05}&\num{6.395974e-06}&\num{1.301611e-06}&\num{1.298960e-06}&\num{1.947616e-07}&\num{1.071790e-07}&\nbf{1.685559e-07}&\numQ{      2.509}\\
\midrule
eoc&                  &\numQ{      2.596}&\numQ{      3.606}&\numQ{      2.297}& \numQ{0.00}  &\numQ{      2.738}&\numQ{      0.862}&\numQ{     -0.6532048}&\\
\midrule
diag&                  &\numQ{1.711864681}&\numQ{2.30884}&\numQ{3.4861319101}&\numQ{4.49190181}&\numQ{2.5568257}&\numQ{0.86382241}&\numQ{      1.630}&\\
      \bottomrule
  \end{tabular}
\end{center}
\caption{$L^2(\Om{N}{h})$ norm difference at $t=T$ between the two methods discussed  in Section \ref{s:mass}.} \vspace*{-0.5cm}
\label{tab:ex3:diff:bdf2}
\end{table}

\subsection{Example 4: Example with topology change}\label{s:numex:extopo}
In this example we consider a geometrically singular configuration.
The level set function to two colliding and afterwards separating circles is
\begin{equation*}
  \phi(\bx,t) = \min( \Vert \bx - s_1(t) \Vert_2, \Vert \bx - s_2(t) \Vert_2) - R, \quad s_1(t) = (0, t - 3/4), \quad s_2(t) = (0, 3/4 - t),
\end{equation*}
where $s_1(t)$ and $s_2(t)$ is the center of the two circles. As the time interval we choose $T=3/2$, so that $\phi(\bx,0)=\phi(\bx,T)$. The corresponding velocity field is discontinuous at $y=0$ and $t=T/2$:
\begin{equation*}
  \bw(x,y,t) = \left\{
    \begin{array}{rl}
      (0,\hphantom{-}1)^T & \text{ if } y > 0 \text { and } t \leq T/2 \text{ or } y < 0 \text { and } t > T/2, \\
      (0,-1)^T & \text{ if } y \leq 0 \text { and } t \leq T/2 \text{ or } y > 0 \text { and } t > T/2.
    \end{array}\right.
\end{equation*}%
\begin{figure}
  \begin{center}
    -1 \includegraphics[width=0.9\textwidth]{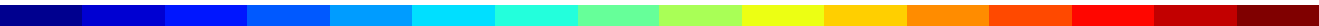} 1 \\[-4.5ex]
    \hspace*{-0.15cm}
  \begin{tabular}{c@{}c@{}c@{}c@{}c@{}c@{}c@{}c@{}c@{}c@{}c}
    \hspace*{-0.41cm}\includegraphics[width=0.26\textwidth,angle=-90]{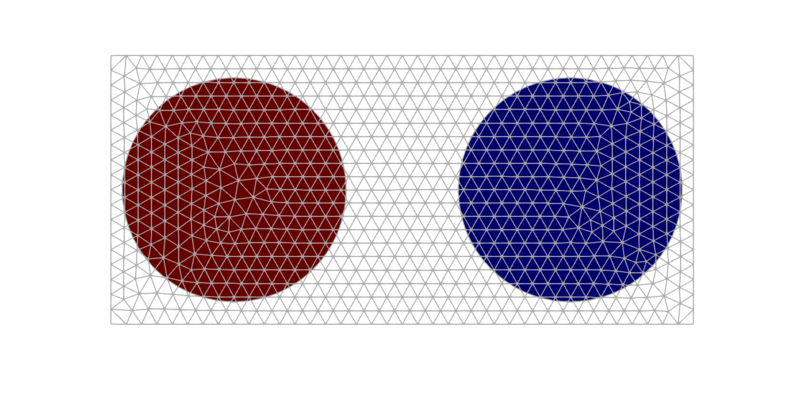} \hspace*{-0.41cm}&
    \hspace*{-0.41cm}\includegraphics[width=0.26\textwidth,angle=-90]{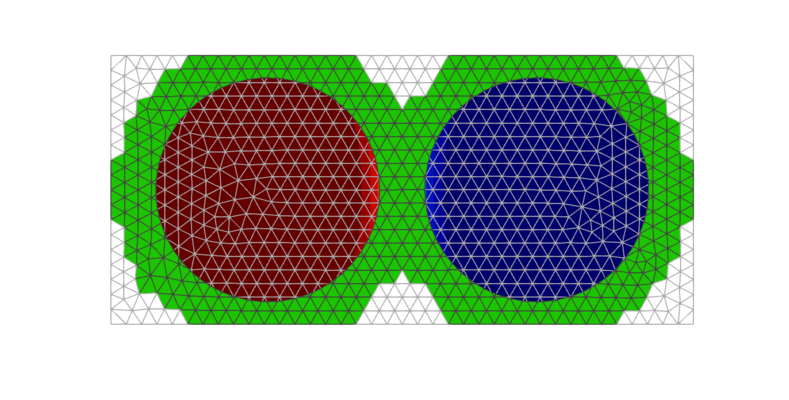} \hspace*{-0.41cm}&
    \hspace*{-0.41cm}\includegraphics[width=0.26\textwidth,angle=-90]{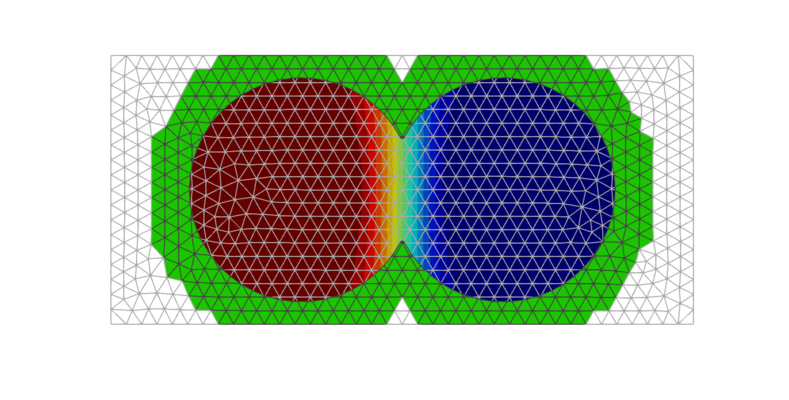} \hspace*{-0.41cm}&
    \hspace*{-0.41cm}\includegraphics[width=0.26\textwidth,angle=-90]{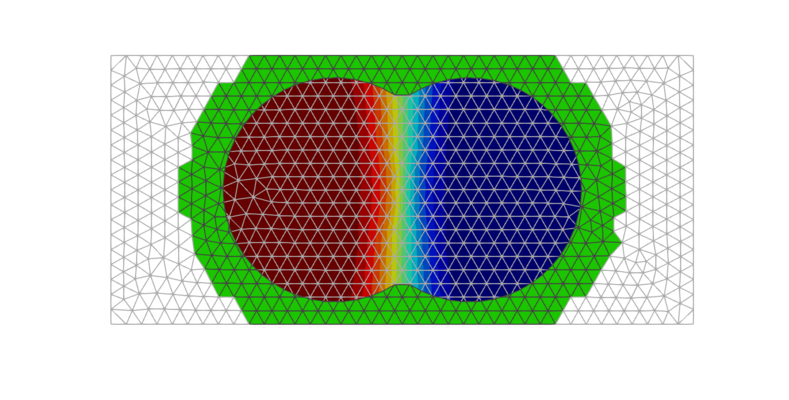} \hspace*{-0.41cm}&
    \hspace*{-0.41cm}\includegraphics[width=0.26\textwidth,angle=-90]{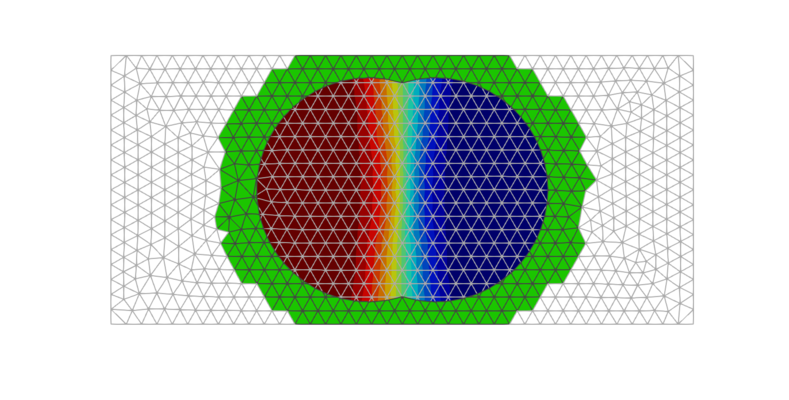} \hspace*{-0.41cm}&
    \hspace*{-0.41cm}\includegraphics[width=0.26\textwidth,angle=-90]{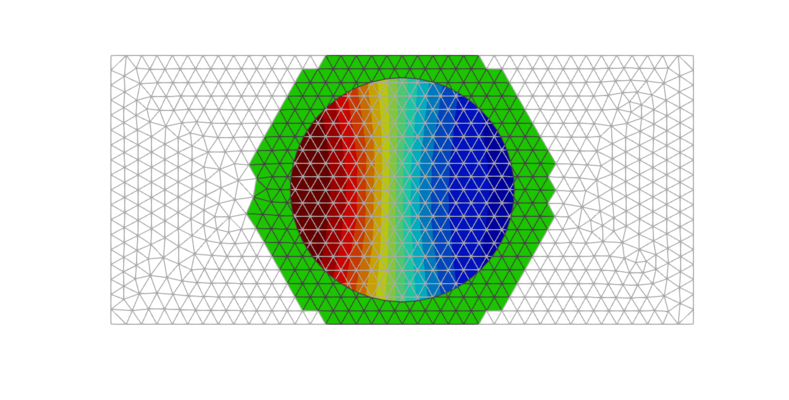} \hspace*{-0.41cm}&
    \hspace*{-0.41cm}\includegraphics[width=0.26\textwidth,angle=-90]{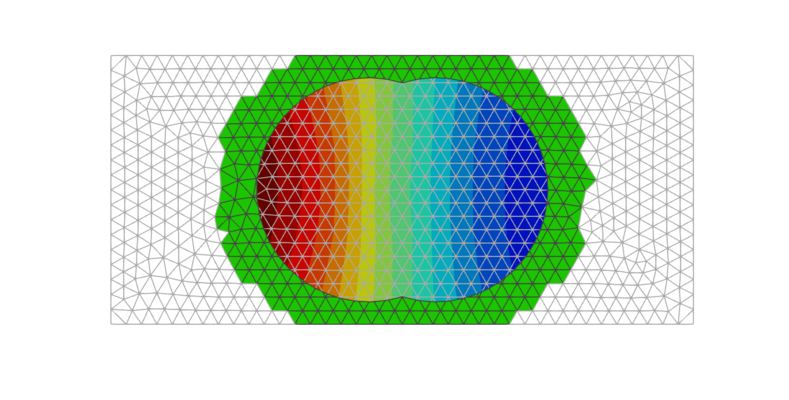} \hspace*{-0.41cm}&
    \hspace*{-0.41cm}\includegraphics[width=0.26\textwidth,angle=-90]{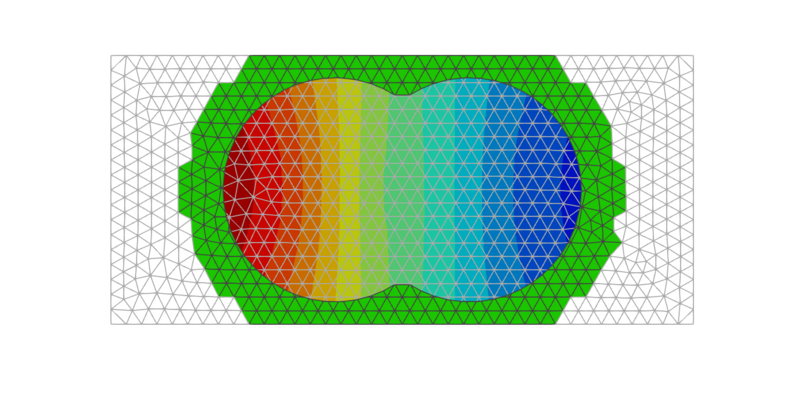} \hspace*{-0.41cm}&
    \hspace*{-0.41cm}\includegraphics[width=0.26\textwidth,angle=-90]{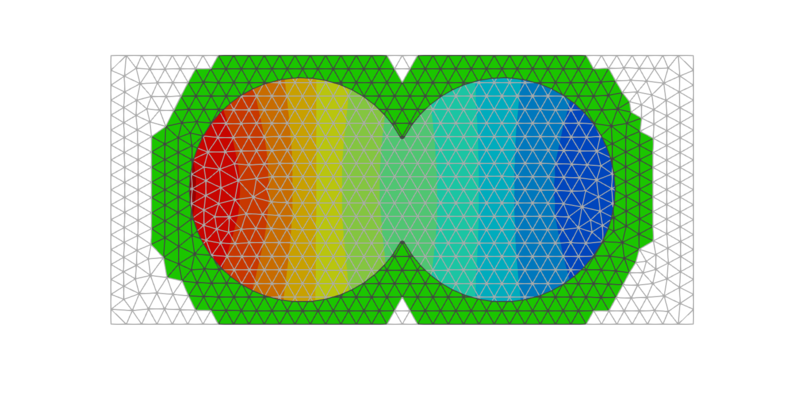} \hspace*{-0.41cm}&
    \hspace*{-0.41cm}\includegraphics[width=0.26\textwidth,angle=-90]{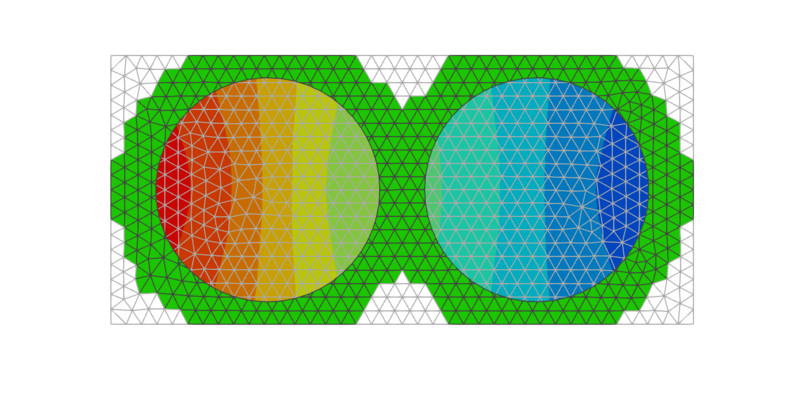} \hspace*{-0.41cm}&
    \hspace*{-0.41cm}\includegraphics[width=0.26\textwidth,angle=-90]{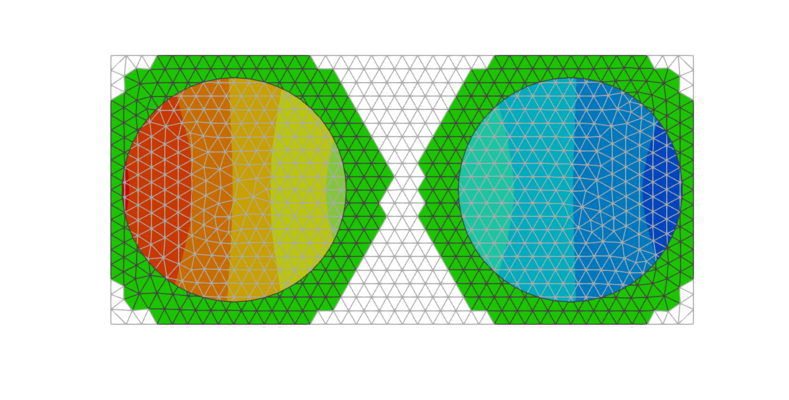} \hspace*{-0.41cm}
    \\[-9ex]
    \hspace*{-0.41cm}\includegraphics[width=0.26\textwidth,angle=-90]{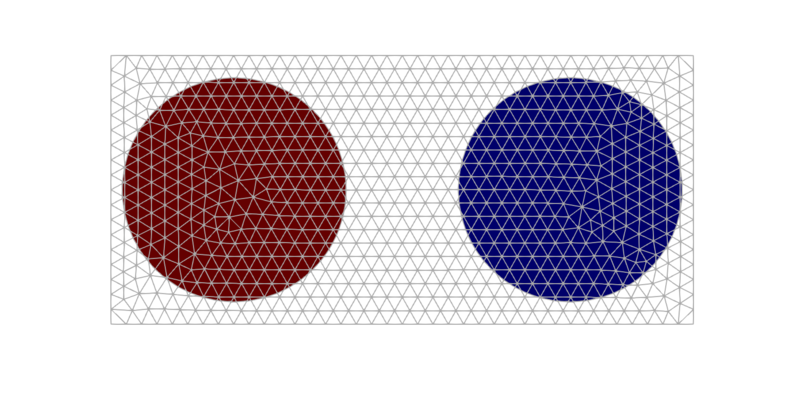} \hspace*{-0.41cm}&
    \hspace*{-0.41cm}\includegraphics[width=0.26\textwidth,angle=-90]{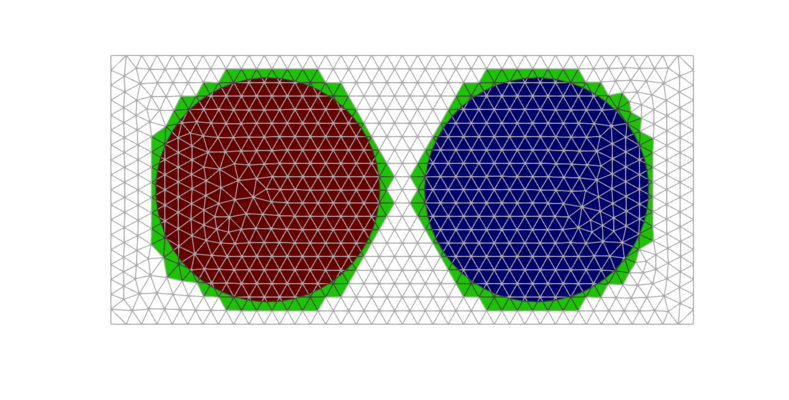} \hspace*{-0.41cm}&
    \hspace*{-0.41cm}\includegraphics[width=0.26\textwidth,angle=-90]{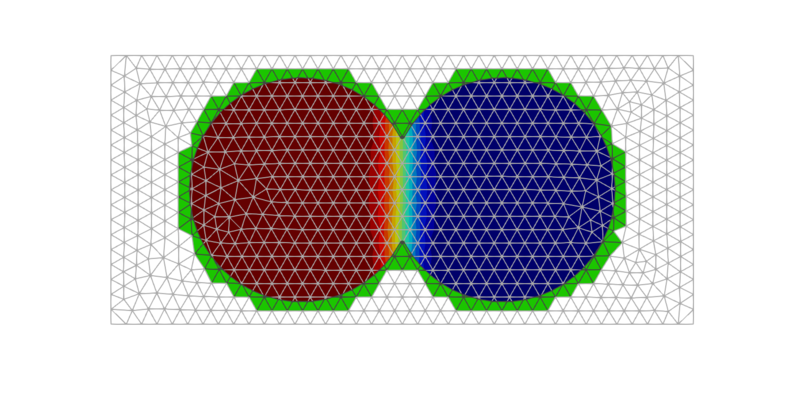} \hspace*{-0.41cm}&
    \hspace*{-0.41cm}\includegraphics[width=0.26\textwidth,angle=-90]{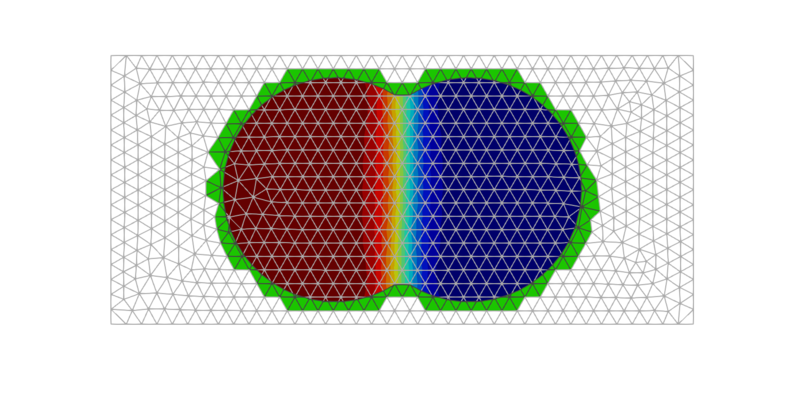} \hspace*{-0.41cm}&
    \hspace*{-0.41cm}\includegraphics[width=0.26\textwidth,angle=-90]{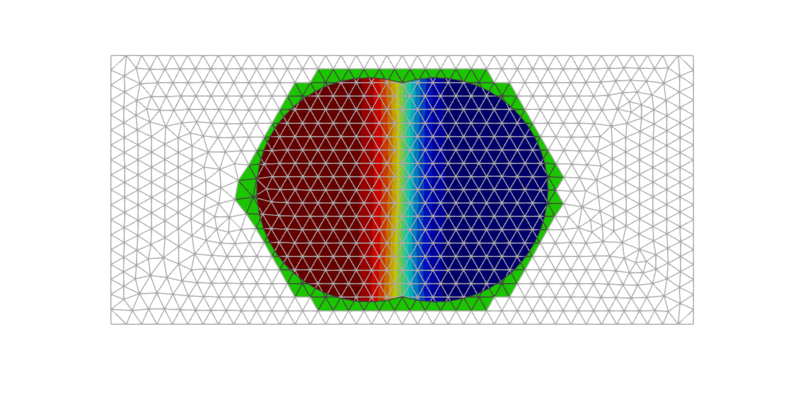} \hspace*{-0.41cm}&
    \hspace*{-0.41cm}\includegraphics[width=0.26\textwidth,angle=-90]{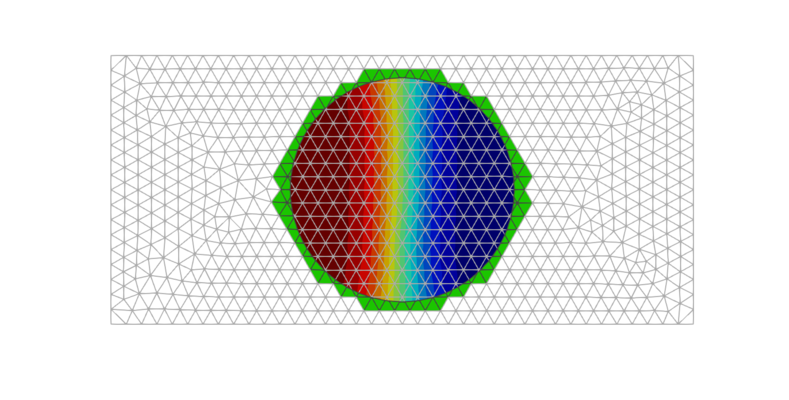} \hspace*{-0.41cm}&
    \hspace*{-0.41cm}\includegraphics[width=0.26\textwidth,angle=-90]{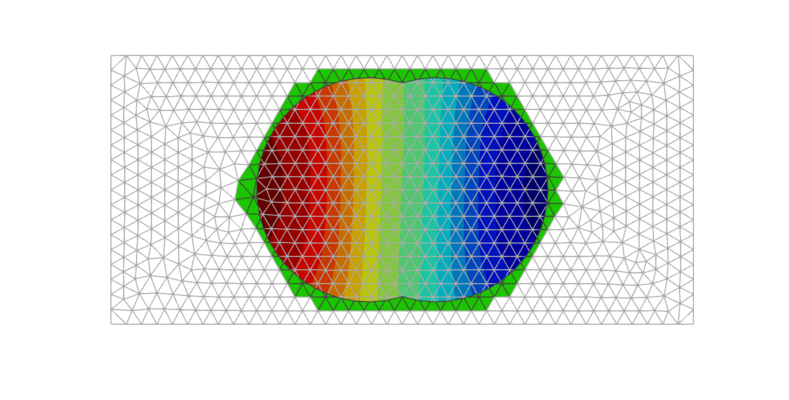} \hspace*{-0.41cm}&
    \hspace*{-0.41cm}\includegraphics[width=0.26\textwidth,angle=-90]{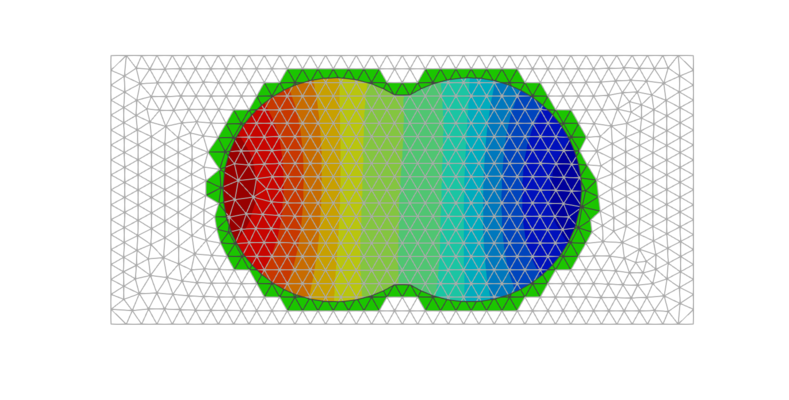} \hspace*{-0.41cm}&
    \hspace*{-0.41cm}\includegraphics[width=0.26\textwidth,angle=-90]{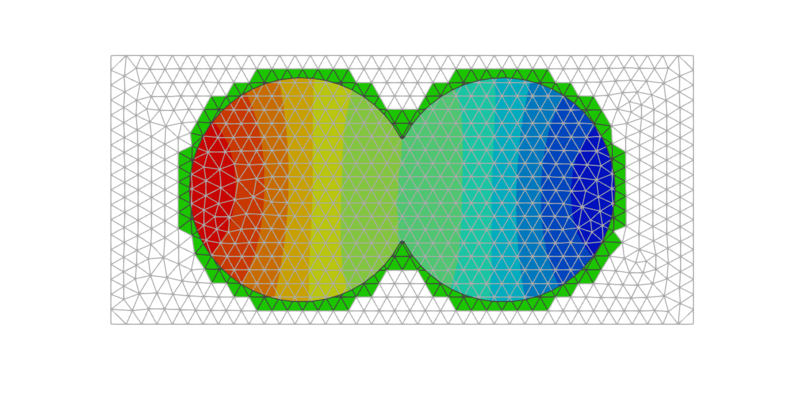} \hspace*{-0.41cm}&
    \hspace*{-0.41cm}\includegraphics[width=0.26\textwidth,angle=-90]{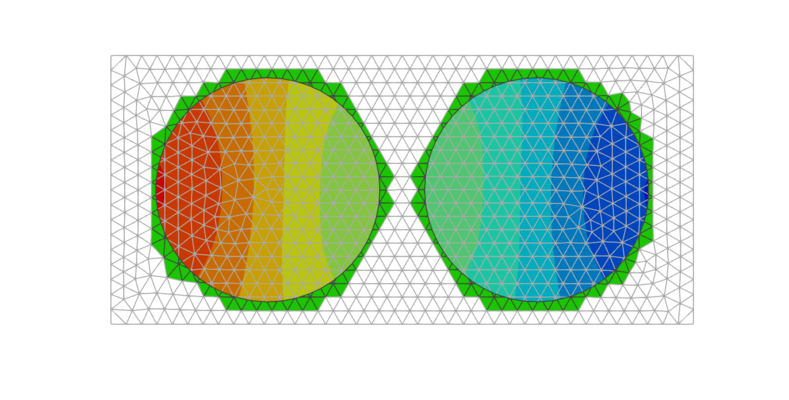} \hspace*{-0.41cm}&
    \hspace*{-0.41cm}\includegraphics[width=0.26\textwidth,angle=-90]{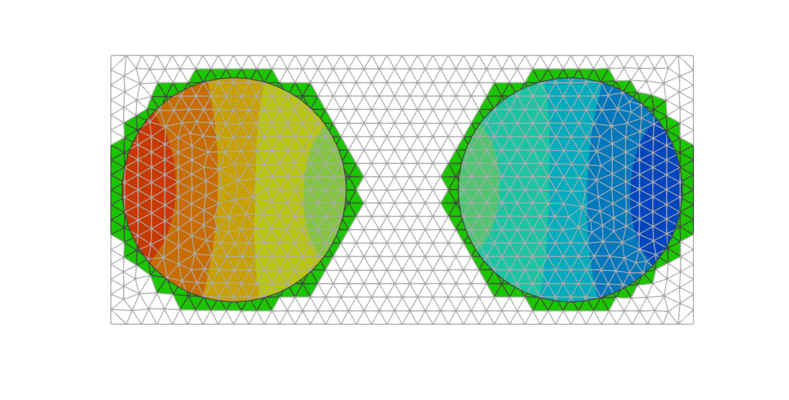} \hspace*{-0.41cm}
    \\[-2ex]
    \small $\!t\!=\!0.0T$ &
    \small $\!t\!=\!0.1T$ &
    \small $\!t\!=\!0.2T$ &
    \small $\!t\!=\!0.3T$ &
    \small $\!t\!=\!0.4T$ &
    \small $\!t\!=\!0.5T$ &
    \small $\!t\!=\!0.6T$ &
    \small $\!t\!=\!0.7T$ &
    \small $\!t\!=\!0.8T$ &
    \small $\!t\!=\!0.9T$ &
    \small $\!t\!=\!1.0T$
    \\[-2ex]
  \end{tabular}
\end{center}
\caption{Simulation results of the example in Section \ref{s:numex:extopo}. The first row displays the results for $\Delta t = T / 10$, the second row the results for $\Delta t = T / 80$. The elements marked in green are the ones in $\Td{n}$.
}
\vspace*{-0.4cm}
\label{fig:topo}
\end{figure}%
As initial concentration we choose $u_0 = -1$ for $y < 0$, i.e. in the lower circle and $u_0 = 1$ for $y > 0$, i.e. in the upper circle. We choose the globally (not component-wise) conservative method \eqref{e:unfFEM1:lag}, $\alpha=0.1$, $h=0.07$ and the time step sizes $\Delta t = T/10$ and $\Delta t = T/80$.
In Figure \ref{fig:topo} the results of the simulations are displayed. Note that this setup is not covered by the analysis as the domain and its evolution is not smooth. Obviously, the method is stable even for this example with non smooth domains. Note that there can be an artificial mass exchange between the two phases if the extension layers intersect before the physical domains intersect, cf. for instance the picture to $\Delta t = T / 10$ and $t=0.1T$. However, for finer time steps these artificial intersections are reduced, cf. the result for $\Delta t = T / 80$ and $t=0.1T$.

\section{Conclusions and open problem} \label{s:Conclusions}
In this paper we introduced a numerical method for solving PDEs on evolving domains. The method is easy to implement as it bases on standard stationary unfitted finite element discretizations and standard finite difference approximations in time. We were able to derive optimal order error bounds in the energy norm for this method. Unlike space--time Galerkin methods, the present approach does not require a physical domain reconstruction on each time slab. In fact,
one only needs approximations of physical domain at time instances $t_n$. No reconstruction of a Lagrangian or arbitrary mapping $\Psi$ from a reference domain is needed either,   which makes the method particularly attractive for application, where the domain deformation is given by a series of snapshots without further information about the underlying motion: One example is the blood flow simulation in a human heart when the patient-specific motion of the heart walls is recovered from a sequence of medical images; see, e.g., \cite{mittal2016computational,su2016cardiac,LOV2017} and references therein.

At the end of this study, let us discuss a few points where extensions and possible modifications of the presented method or its analysis are worth pursuing.

So far, in the method we used an extension to the domain $\Om{n}{h}$ based on \textit{a priori} estimated strip size. One could improve this by only involving elements that are relevant on the next time step based on $\Om{n+1}{h}$ if this information is available. Furthermore, one could separate the stabilized solution step and the extension into two steps as has been done in the semi-discrete method.

In the analysis we only derived error estimates for the $H^1$ norm in space, but observed a higher order convergence (in space) in the $L^2$ norm. Using duality techniques we expect that improved rates can also be obtained for $L^2$ norm estimates.

In this paper we treated only implicit Euler discretizations in the analysis and commented on extensions to BDF2 discretization which we also used in the numerical experiments. An extension to more general time stepping scheme has not been used so far, but is an interesting topic.

We only consider a comparably simple model problem.
Many applications will involve more complex problems, e.g. two-phase Navier-Stokes equations. An extension of the method to these problems should be investigated and analyzed in the future.

For the numerical examples we used a geometry approximation and finite element order $q=m=1$ although the analysis allows also for higher order schemes in space. Due to practical reasons such as accurate and robust numerical integration, the development and implementation of higher order methods can be difficult, but should not -- based on recently developed techniques \cite{muller2013highly,lehrenfeld2015cmame,saye2015hoquad,fries2015} -- pose a major obstacle.

\appendix
\section{Proof of Lemma \ref{L_exch}}
\begin{proof}
The result follows in three steps.

\emph{Step 1.} Define a sequence of $v^m\in C^\infty(Q_0)$ such that $v^m\to v$ in $L^2(Q_0)$ and   $v^m_t\to v_t$ in $L^2(Q_0)$. Functions $v_m$ can be constructed by the following standard argument. First note that from inequality $v^2(0)=
v^2(t)-\int_{0}^{t}(v^2)_t\,ds\le v^2(t)+2(\int_{0}^{T}v^2\,ds)^{\frac12}(\int_{0}^{T}v^2_t\,ds)^{\frac12}$ it follows that  $v(x,0)$ is well defined as an element of $L^2(\Omega_0)$. Similar $v(T)\in L^2(\Omega_0)$. Thus we consider
\[
\tilde v(x,t)=\left\{
\begin{split}
   v(x,t) &~~\text{for}~ x\in\Om{}{0},\,t\in(0,T) \\
   v(x,0) &~~\text{for}~ x\in\Om{}{0},\,t\le0 \\
   v(x,T) &~~\text{for}~ x\in\Om{}{0},\,t\ge T \\
    0 &~~ \text{otherwise}
\end{split}
\right.,
\]
and define $v^m= \omega_{\epsilon_m}*\tilde{v}$ with a smooth mollifier $\omega_{\epsilon}$ and $\epsilon_m\to0$ with $m\to\infty$. By the basic properties of mollifiers $v^m_t= \omega_{\epsilon_m}*\tilde{v}_t$, and  $v^m\in C^\infty(Q_0)$ is the desired sequence. Due to the continuity of $\E{0}\,:\,L^2(\Omega_0)\to L^2(\mathcal{O}(\Omega_0))$, it holds
\[
\|\E{0}w\|^2_{L^2(\mathcal{O}(Q_0))}=\int_{0}^{T}\int_{\mathcal{O}(\Omega_0)}|\E{0}w|^2\,dx\,dt\le c \int_{0}^{T}\int_{\Omega_0}w^2\,dx\,dt=
c\|w\|^2_{L^2(Q_0)}\quad\text{for any}~w\in L^2(Q_0).
\]
Since $\E{0}$ is linear, we infer that $\E{0}$ is continuous from $L^2(Q_0)$ to $L^2(\mathcal{O}(Q_0))$ and hence the convergence
 $v^m\to v$ in $L^2(Q_0)$ and   $v^m_t\to v_t$ in $L^2(Q_0)$ imply $\E{0}v^m\to \E{0}v$ in $L^2(\mathcal{O}(Q_0))$ and   $\E{0}v^m_t\to \E{0}v_t$ in $L^2(\mathcal{O}(Q_0))$.

\emph{Step 2.} We now show that for a smooth function $w=v^m\in C^\infty(Q_0)$ ($m$ is fixed) the extension and time derivative commute, i.e., $\left(\E{0} w\right)_t=\E{0} w_t\quad\text{in}~~\mathcal{O}(Q_0)$. 
We have for fixed $t\in(0,T)$ and $|\delta|$ sufficiently small:
\[
\begin{split}
\E{0}w(t)-\E{0}w(t+\delta)&=\E{0}(w(t)-w(t+\delta))\\
&=\E{0}(\delta w_t(t)-\xi_{\delta})\quad\text{with}~\xi_{\delta}(\cdot)=\int_{t}^{t+\delta}w_{tt}(s,\cdot)s\,ds\\
&=\delta\E{0}w_t(t)-\E{0}\xi_{\delta}\quad \text{in}~~\mathcal{O}(\Omega_0),
\end{split}
\]
where we used the linearity of $\E{0}$. 
By the continuity of $\E{0}$ in $L^\infty(\Omega_0)$, it holds
\[
\|\E{0}\xi_{\delta}\|_{L^\infty(\mathcal{O}(\Omega_0))}\le c\|\xi_{\delta}\|_{L^\infty(\Omega_0)}\le c|\delta|^2\|w_{tt}\|_{L^\infty(Q_0)}\le C|\delta|^2,
\]
where $C$ is independent of $\delta$.
Since $t$ was taken arbitrary from $(0,T)$,  this proves $\left(\E{0} w\right)_t=\E{0} w_t$ in $\mathcal{O}(Q_0)$.

\emph{Step 3.} Finally, we show $\left(\E{0} v\right)_t=\E{0} v_t$ by a density argument. Using $L^2(\mathcal{O}(Q_0))$-convergence from step 1 and the commutation property from step 2, we get for any finite function $\eta\in \dot{C}(\mathcal{O}(Q_0))$:
\[
\begin{split}
\int_{\mathcal{O}(Q_0)} (\E{0} v)\eta_t \, d(x,t) &= \lim_{m\to\infty}\int_{\mathcal{O}(Q_0)} (\E{0} v^m)\eta_t \, d(x,t)
= -\lim_{m\to\infty}\int_{\mathcal{O}(Q_0)} (\E{0} v^m)_t\eta \, d(x,t)\\
&= -\lim_{m\to\infty}\int_{\mathcal{O}(Q_0)} (\E{0} v^m_t)\eta \, d(x,t)= -\int_{\mathcal{O}(Q_0)} (\E{0} v_t)\eta \, d(x,t).
\end{split}
\]
Thus $\left(\E{0} v\right)_t=\E{0} v_t$ holds by the definition of the weak partial derivative.
\end{proof}

\bibliography{literatur}{}
\bibliographystyle{siam}
\end{document}